\documentclass[reqno,centertags,12pt, draft]{amsart}
\usepackage{amsmath,amsthm,amscd,amssymb}
\usepackage{latexsym}
\usepackage{verbatim}
\newcommand{\bbC}{{\mathbb{C}}}
\newcommand{\bbD}{{\mathbb{D}}}

\newcommand{\bbN}{{\mathbb{N}}}

\newcommand{\bbR}{{\mathbb{R}}}

\newcommand{\bbZ}{{\mathbb{Z}}}

\newcommand{\calE}{{\mathcal{E}}}

\newcommand{\calH}{{\mathcal H}}
\newcommand{\calI}{{\mathcal I}}

\newcommand{\calK}{{\mathcal K}}

\let\det=\undefined\DeclareMathOperator{\det}{det}

\newcommand{\dott}{\,\cdot\,}
\newcommand{\no}{\nonumber}
\newcommand{\lb}{\label}
\newcommand{\f}{\frac}

\newcommand{\tr}{\text{\rm{Tr}}}

\newcommand{\rank}{\text{\rm{rank}}}
\newcommand{\ran}{\text{\rm{ran}}}
\newcommand{\dom}{\text{\rm{dom}}}

\newcommand{\ac}{\text{\rm{ac}}}

\newcommand{\supp}{\text{\rm{supp}}}

\newcommand{\bi}{\bibitem}

\newcommand{\beq}{\begin{equation}}
\newcommand{\eeq}{\end{equation}}
\newcommand{\ba}{\begin{align}}
\newcommand{\ea}{\end{align}}
\newcommand{\veps}{\varepsilon}
\newcounter{smalllist}
\newenvironment{SL}{\begin{list}{{\rm\roman{smalllist})}}{%
\setlength{\topsep}{0mm}\setlength{\parsep}{0mm}\setlength{\itemsep}{0mm}%
\setlength{\labelwidth}{2em}\setlength{\leftmargin}{2em}\usecounter{smalllist}%
}}{\end{list}}

\DeclareMathOperator{\Ima}{Im}

\DeclareMathOperator*{\wlim}{w-lim}

\DeclareMathOperator{\diag}{diag}
\allowdisplaybreaks
\numberwithin{equation}{section}
\newtheorem{theorem}{Theorem}[section]

\newtheorem*{p2.1}{Proposition 2.1}
\newtheorem*{OQ2}{Open Question 11.2}
\newtheorem*{OQ3}{Open Question 11.3}
\newtheorem*{OQ4}{Open Question 11.4}

\newtheorem{proposition}[theorem]{Proposition}
\newtheorem{lemma}[theorem]{Lemma}
\newtheorem{corollary}[theorem]{Corollary}
\theoremstyle{definition}

\newtheorem{conjecture}[theorem]{Conjecture}

\theoremstyle{remark}
\newtheorem*{remark}{Remark}
\newtheorem*{remarks}{Remarks}

\newcommand{\abs}[1]{\lvert#1\rvert}

\begin{document}

\title[On the Koplienko Spectral Shift Function, I.\ Basics]{On the Koplienko Spectral Shift \\ Function, I. Basics}
\author[F.~Gesztesy, A.~Pushnitski, and B.~Simon]{Fritz~Gesztesy$^1$,
Alexander Pushnitski$^2$, and Barry Simon$^3$}

\thanks{$^1$ Department of Mathematics, University of Missouri, Columbia, MO 65211, USA.
E-mail: fritz@math.mizzou.edu. Supported in part by NSF Grant DMS-0405526}

\thanks{$^2$ Department of Mathematics, King's College London, Strand, London WC2R 2LS, England,
UK. E-mail: alexander.pushnitski@kcl.ac.uk. Supported in part by the Leverhulme Trust.}

\thanks{$^3$ Mathematics 253-37, California Institute of Technology, Pasadena, CA 91125, USA.
E-mail: bsimon@caltech.edu. Supported in part by NSF Grant DMS-0140592 and
U.S.--Israel Binational Science Foundation (BSF) Grant No.\ 2002068}

\thanks{Submitted to the Marchenko and Pastur birthday issue of Journal of Mathematical Physics, Analysis and Geometry}

% \date{February, 2007}
\date{May 16, 2007}

%%%%%%%%%%%%%%%%%%%%%%
\begin{abstract}
We study the Koplienko Spectral Shift Function (KoSSF), which is distinct from
the one of Krein (KrSSF). KoSSF is defined for pairs $A,B$ with $(A-B)\in\calI_2$,
the Hilbert--Schmidt operators, while KrSSF is defined for pairs $A,B$ with
$(A-B)\in\calI_1$, the trace class operators. We review various aspects of
the construction of both KoSSF and KrSSF. Among our new results are: (i) that
any positive Riemann integrable function of compact support occurs as a KoSSF;
(ii) that there exist $A,B$ with $(A-B)\in\calI_2$ so $\det_2 ((A-z)(B-z)^{-1})$
does not have nontangential boundary values; (iii) an alternative definition of KoSSF
in the unitary case; and (iv) a new proof of the invariance of the a.c.\ spectrum
under $\calI_1$-perturbations that uses the KrSSF.
\end{abstract}
%%%%%%%%%%%%%%%%%%%%%%

\maketitle

%%%%%%%%%%%%%%%%%%%%%%%%%%%%%%%
\section{Introduction} \lb{s1}
%%%%%%%%%%%%%%%%%%%%%%%%%%%%%%%

In 1941, Titchmarsh \cite{Ti41} (see also \cite[pp.\ 1564--1566]{DS88} for the result)
proved that if
$$
V\in L^1 ((0,\infty);dx), \, \text{ $V$ real-valued},
$$
and
\begin{align} \lb{1.1}
& H_{\theta} = -\f{d^2}{dx^2} + V,  \\
& \dom(H_{\theta})=\{f \in L^2((0,\infty);dx) \,|\, f, f' \in AC([0,R]) \text{ for all $R>0$};  \no \\
& \hspace*{1cm} \sin(\theta) f'(0)+ \cos(\theta) f(0)=0; \, (-f''+Vf)  \in L^2((0,\infty);dx)\},  \no
\end{align}
for some $\theta \in [0,\pi)$, then
$$
\sigma_\ac (H_{\theta})=[0,\infty).
$$
(Actually, he explicitly computed the spectral function in terms of the inverse square
of the modulus of the Jost function for positive energies.) It was later realized that
the a.c.\ invariance, that is,
\begin{equation} \lb{1.3}
\sigma_\ac (H_{\theta})=\sigma_\ac (H_{0,\theta})
\end{equation}
with
\begin{align} \lb{1.2}
& H_{0,\theta} = -\f{d^2}{dx^2},  \\
& \dom(H_{0,\theta})=\{f \in L^2((0,\infty);dx) \,|\, f, f' \in AC([0,R]) \text{ for all $R>0$};  \no \\
& \hspace*{3cm} \sin(\theta) f'(0)+ \cos(\theta) f(0)=0; \, f''  \in L^2((0,\infty);dx)\}, \no
\end{align}
is a special case of an invariance of the absolutely continuous spectrum,
$\sigma_\ac(\cdot)$ for the passage from $A$ to $B$ if $(A-B) \in \calI_1$, the trace class. 
In the present context of the pair $(H_{\theta},H_{0,\theta})$ one has 
$[(H_{\theta}+E)^{-1} - (H_{0,\theta} +E)^{-1}]\in\calI_1$ for $E>0$
sufficiently large. The abstract trace class result is associated with Birman
\cite{Bir62,Bir63}, Kato \cite{Kato1,Kato2}, and Rosenblum \cite{Ros57}.

Our original and continuing motivation is to find a suitable operator theoretic result
connected with the remarkable discovery of Deift--Killip \cite{DeiftK} that for the above
\eqref{1.1}/\eqref{1.2} case, one has \eqref{1.3} if one only assumes $V\in L^2
((0,\infty);dx)$. Note that $V\in L^2((0,\infty);dx)$ implies that
$$
[(H_{\theta}+E)^{-1}-(H_{0,\theta}+E)^{-1}]\in \calI_2,
$$
the Hilbert--Schmidt class. However, there is no totally general invariance result for
a.c.\ spectrum under non-trace class perturbations: It is a result of Weyl \cite{We09}  
and von Neumann \cite{vNeu} that given any self-adjoint $A$, there is a $B$ with 
pure point spectrum and $(A-B)\in\calI_2$. Kuroda \cite{Kur} extends this to $\calI_p$, 
$p\in(1,\infty)$, the trace ideals. Thus, we seek general operator criteria on when
$(A-B)\in\calI_2$ but \eqref{1.3} still holds.

We hope such a criterion will be found in the spectral shift function of Koplienko \cite{Kop}
(henceforth KoSSF), an object which we believe has not received the attention it deserves.
One of our goals in the present paper is to make propaganda for this object.

Two references for trace ideals we quote extensively are Gohberg--Krein \cite{GK} and
Simon \cite{STI}. We follow the notation of \cite{STI}. Throughout this paper all Hilbert
spaces are assumed to be complex and separable.

The KoSSF, $\eta (\lambda;A,B)$, is defined when $A$ and $B$ are bounded self-adjoint
operators satisfying $(A-B) \in\calI_2$, and is given by
\begin{equation} \lb{1.4}
\int_{\bbR} f''(\lambda) \eta(\lambda;A,B) \, d\lambda = \tr \bigg( f(A)-f(B)-
 \f{d}{d\alpha} f(B+\alpha (A-B))\bigg|_{\alpha=0}\bigg),
\end{equation}
where the right-hand side is sometimes (certainly if $(A-B)\in\calI_1$) the simpler-looking
\begin{equation} \lb{1.4a}
\tr (f(A)-f(B)-(A-B) f'(B)).
\end{equation}
$\eta$ has two critical properties: $\eta\in L^1(\bbR)$ and $\eta\geq 0$.
We mainly consider bounded $A,B$ here,
but see the remarks in Section~\ref{s9}.

Formula \eqref{1.4} requires some assumptions on $f$.
In Koplienko's original paper \cite{Kop} the case $f(x)=(x-z)^{-1}$ was considered
and then \eqref{1.4} was extended to the class of rational functions with poles off
the real axis. Later, Peller \cite{Pe05} extended the class of functions $f$ and
found sharp sufficient conditions on $f$
which guarantee that \eqref{1.4} holds. These conditions were stated
in terms of Besov spaces. Essentially, Peller's construction requires that
\eqref{1.4} hold for some sufficiently wide class of functions, so that this
class is dense in a certain Besov space, and then provides an
extension onto the whole of this Besov space.

We will use this aspect of Peller's work and will not worry about the classes of
$f$ in this paper. For the most part we will work with $f\in C^\infty(\bbR)$
and Peller's construction provides an extension to a wider function class.

The model for the KoSSF is, of course, the spectral shift function of Krein (henceforth KrSSF), 
denoted by $\xi (\lambda;A,B)$, and defined for $A,B$ with $(A-B)\in\calI_1$ by 
\begin{equation} \lb{1.5}
\int_{\bbR} \xi (\lambda; A,B) f'(\lambda) \, d\lambda = \tr (f(A)-f(B)).
\end{equation}
In the appendix, we recall a quick way to define $\xi$, its main properties and, most importantly,
present an  argument that shows how it can be used to derive the invariance of a.c.\ spectrum
without recourse to scattering theory.

As we will see in Section~\ref{s2}, it is easy to construct analogs of $\eta$ for any
$\calI_n$, $n\in\bbN$, but they are only tempered distributions. What makes $\eta$ different
is its positivity, which also implies it lies in $L^1(\bbR)$ (by taking $f$ suitably).
This positivity should be thought of as a general convexity result---something
hidden in Koplienko's paper \cite{Kop}.

One of our goals here is to emphasize this convexity. Another is to present a ``baby"
finite-dimensional version of the double Stieltjes operator integral of Birman--Solomyak
\cite{BS1,BS2,BS4}, essentially due to L\"owner \cite{Loe1}, whose contribution here
seems to have been overlooked.

In Section~\ref{s2}, we define $\eta$ when $(A-B)$ is trace class,
and in Section~\ref{s3},
we discuss the convexity result that is equivalent to positivity of $\eta$. In Section~\ref{s4},
we prove a lovely bound of Birman--Solomyak \cite{BS2,BS4}:
\begin{equation} \lb{1.6}
\|f(A)-f(B)\|_{\calI_2} \leq \|f'\|_\infty \|A-B\|_{\calI_2}.
\end{equation}
Here and in the remainder of this paper $\|\cdot\|_{\calI_p}$ denotes the norm in
the trace ideals $\calI_p$, $p\in[1,\infty)$. In Section~\ref{s5}, we use \eqref{1.6}
plus positivity of $\eta$ to complete the construction of $\eta$.

We want to emphasize an important distinction between the KrSSF and the KoSSF. The former
satisfies a chain rule
\begin{equation} \lb{1.7}
\xi(\dott; A,C) = \xi (\dott; A,B) + \xi (\dott; B,C),
\end{equation}
while $\eta$ instead satisfies a corrected chain rule
\begin{equation} \lb{1.7a}
\eta (\dott;A,C)=\eta (\dott;A,B) + \eta(\dott; B,C) + \delta \eta(\dott; A,B,C),
\end{equation}
where $\delta\eta$ satisfies
\begin{equation} \lb{1.8}
\int_{\bbR} g'(\lambda) \delta\eta (\lambda;A,B)\, d\lambda = \tr ((A-B)(g(B)-g(C))).
\end{equation}
(Here $g$ corresponds to $f'$ when comparing with \eqref{1.4}--\eqref{1.5}.)
It is in estimating \eqref{1.8} that \eqref{1.6} will be critical.

We view Sections~\ref{s2}--\ref{s5} as a repackaging in a prettier ribbon of
Koplienko's construction in \cite{Kop}. Section~\ref{s5a} explores what $\eta$'s can
occur. In Sections~\ref{new-s7} and \ref{s6}, we discuss the connection
to $\det_2(\cdot)$ and present a new result: an example of $(A-B)\in\calI_2$
where $\det_2 ((A-z)(B-z)^{-1})$ does not have nontangential limits to the real axis
a.e.\ This is in contradistinction to the KrSSF,
where $(A-B)\in\calI_1$ implies $\det((A-z)(B-z)^{-1})$ has a nontangential limit
$z\to\lambda$ for a.e.\ $\lambda\in\bbR$. The latter is a consequence of the formula
\begin{equation}
\log(\det((A-z)(B-z)^{-1})) = \int_{\bbR} (\lambda -z)^{-1} \xi(\lambda;A,B)\,
d\lambda,  \quad z\in\bbC\backslash \bbR,   \lb{1.9}
\end{equation}
since the right-hand side of \eqref{1.9} represents a difference of two Herglotz functions.

Sections~\ref{s9} and \ref{s10} discuss extensions of $\eta$ to the case of unbounded
operators with a trace class condition on the resolvents and to unitary operators.
Here a key is that $\eta$ is not determined until one makes a choice of interpolation.
Section~\ref{s11} discusses some conjectures.

In a future joint work, we will explore what one can learn about the KoSSF from Szeg\H{o}'s
theorem \cite{OPUC}, the work of Killip--Simon \cite{KS} and of Christ--Kiselev \cite{CK}.
This will involve the study of $\eta$ for suitable Schr\"odinger operators and Jacobi
and CMV matrices for perturbations in $L^p$, respectively, $\ell^p$, $p\in [1,2)$.

\medskip

We are indebted to E.~Lieb, K.~A.~Makarov, V.~V.~Peller, and M.~B.~Ruskai
for useful discussions. F.~G. and A.~P.~ wish to thank Gary Lorden and Tom Tombrello
for the hospitality of Caltech where some of this work was done. F.~G. gratefully
acknowledges a research leave for the academic year 2005/06 granted by the
Research Council and the Office of Research of the University of Missouri-Columbia.
A.~P. gratefully acknowledges financial support by the Leverhulme Trust.

\smallskip
It is a great pleasure to dedicate this paper to the birthdays of two giants of spectral theory:
Vladimir A.~Marchenko and Leonid A.~Pastur.

%%%%%%%%%%%%%%%%%%%%%%%%%%%%%%%%%%%%%%%
\section{The KoSSF $\eta(\dott;A,B)$ in the Trace Class Case} \lb{s2}
%%%%%%%%%%%%%%%%%%%%%%%%%%%%%%%%%%%%%%%

We begin with what can be said of $\calI_n$ perturbations, $n\in\bbN$, and then
turn to what is special for $n=1,2$. We note that our approach has common elements to
the one used by Dostani\'c \cite{Do93}.

%%%%%%%%%%%%%%%%%
\begin{proposition}\lb{P2.1} Let $A,B$ be bounded self-adjoint operators with
\begin{equation} \lb{2.1a}
A=B+X.
\end{equation}
For $\alpha,t\in\bbR$, define
\[
f_t(\alpha) =e^{it(B+\alpha X)}.
\]
Then $f_t(\alpha)$ is $C^\infty$ in $\alpha$ and
\begin{equation} \lb{2.1}
\f{d^k f_t}{d\alpha^k} =i^k \int_{\substack{0<s_j<t \\ \sum_{j=1}^{k} s_j <t}}
f_{s_1}(\alpha) X f_{s_2}(\alpha) \dots f_{s_k}(\alpha) X
f_{t-s_1- \dots -s_k}(\alpha)\, ds_1 \dots ds_k.
\end{equation}
If $X\in\calI_p$ for $p\geq k$, $k\in\bbN$, then $d^k f_t/d\alpha^k\in\calI_{p/k}$ and
\begin{equation} \lb{2.2}
\biggl\| \f{d^k f_t}{d\alpha^k}\biggr\|_{\calI_{p/k}} \leq \f{t^k}{k!}\, \|X\|_{\calI_p}^k.
\end{equation}
In particular, if $n\in\bbN$ and $X\in\calI_n$, then
\begin{gather}
g_t(A,B)\equiv \bigg(e^{itA}-e^{itB} - \sum_{k=1}^{n-1} \f{1}{k!}
\bigg(\f{d}{d\alpha}\bigg)^k f_t(\alpha)\bigg|_{\alpha=0}\bigg) \in\calI_1, \lb{2.3} \\
\|g_t(A,B)\|_{\calI_1} \leq \f{t^n}{n!}\, \|X\|_{\calI_n}^n.    \lb{2.4}
\end{gather}
\end{proposition}
%%%%%%%%%%%%%%%%
\begin{proof} For $k=1$, \eqref{2.1} comes from taking limits in DuHamel's formula
\[
e^C-e^D =\int_0^1 e^{\beta C} (C-D) e^{(1-\beta)D}\, d\beta.
\]
The general $k$ case then
follows by induction.

\eqref{2.1} implies \eqref{2.2} by H\"older's inequality for operators (see
\cite[p.\ 21]{STI}). \eqref{2.3} is then Taylor's theorem with remainder and \eqref{2.4}
follows from \eqref{2.2}.
\end{proof}
%%%%%%%%%%%%%%%%

%%%%%%%%%%%%%%%%
\begin{theorem}\lb{T2.2} Let $A,B$ be bounded self-adjoint operators such that
$X=(A-B)\in\calI_n$ for some $n\in\bbN$. Let $f$ be of compact support
with $\widehat f$, its Fourier transform, satisfying
\begin{equation} \lb{2.5}
\int_{\bbR} (1+\abs{k})^n \abs{\widehat f(k)}\, dk <\infty.
\end{equation}
Then,
\begin{equation} \lb{2.6x}
\bigg( f(A)-f(B) - \sum_{j=1}^{n-1} \frac1{k!}\bigg( \f{d}{d\alpha}\bigg)^k
f(B+\alpha X)\bigg|_{\alpha=0}\bigg) \in\calI_1,
\end{equation}
and there is a distribution $T$ with
\begin{equation} \lb{2.7a}
\tr\bigg(f(A)-f(B) - \sum_{j=1}^{n-1} \frac1{k!}\bigg( \f{d}{d\alpha}\bigg)^k
f(B+\alpha X)\bigg|_{\alpha=0} \bigg)=\int_{\bbR} T(\lambda) f^{(n)}(\lambda)\, d\lambda.
\end{equation}
Moreover, the distribution $T$ is such that $\widehat T \in L^\infty(\bbR;dt)$.
\end{theorem}
%%%%%%%%%%%%%%%%%
\begin{proof} This is immediate from the estimates in Proposition~\ref{P2.1} and
\[
f(A)=(2\pi)^{-1/2} \int_{\bbR} \widehat f(t)\, e^{itA}\, dt.
\]

For, by \eqref{2.4}, we have
\begin{equation}
\|\text{LHS of \eqref{2.6x}}\|_{\calI_1} \leq C\int_{\bbR} \abs{t}^n \abs{\widehat f(t)}\,
dt = C \int_{\bbR} \abs{\widehat{f^{(n)}}(t)}\, dt.
\end{equation}
Thus, \eqref{2.7a} defines a distribution $T$ with
\[
\abs{T(f)} \leq C \int_{\bbR} \abs{\widehat f(t)}\, dt,
\]
so $\widehat T$ is a function in $L^\infty(\bbR;dt)$.
\end{proof}
%%%%%%%%%%%%%%%%%

Notice that, as we have seen,
\[
\f{d}{d\alpha}\, e^{it(B+\alpha X)}\bigg|_{\alpha=0} =
i\int_0^t e^{i \beta B} Xe^{i (t-\beta) B}\, d\beta
\]
so that formally,
\[
 \tr\bigg(\f{d}{d\alpha}\, f(B+\alpha X)\bigg|_{\alpha=0} -Xf'(B)\bigg) =0,
 \]
and formally, \eqref{1.4a} can replace the right-hand side of \eqref{1.4}. This can be
proven if the commutator $[B,X]=[B,A]$ is trace class.

Dostani\'c \cite[Theorem~2.9]{Do93} essentially proves that $T$ is the derivative of an 
$L^2(\bbR;d\lambda)$-function.

A key point for us is that in the case $n=2$, the distribution $T$ is given by an
$L^1(\bbR;d\lambda)$-function. We start
this construction here by considering the trace class case.

%%%%%%%%%%%%%%%
\begin{lemma}\lb{L2.3} Let $B$ be a self-adjoint operator and let $X=(A-B)\in\calI_1$.
Then there is a {\rm{(}}complex\,{\rm{)}} measure $d\mu_{B,X}$ on $\bbR$ such that
for any bounded Borel function, $f$,
\begin{equation} \lb{2.6}
\tr (Xf(B))=\int_{\bbR} f(\lambda)\, d\mu_{B,X}(\lambda).
\end{equation}
Equation \eqref{2.6} defines $d\mu_{B,X}$ uniquely.
\end{lemma}
%%%%%%%%%%%%%%%
\begin{proof} Equation \eqref{2.6} yields uniqueness of the measure $d\mu_{B,X}$ since it
defines the integral for all continuous functions $f$. Regarding existence of $d\mu_{B,X}$,
the spectral theorem asserts the existence of measures $d\mu_{B;\varphi,\psi}$, such that
\begin{equation} \lb{2.7}
\langle\varphi, f(B)\psi\rangle = \int_{\bbR} f(\lambda)\, d\mu_{B;\varphi,\psi}(\lambda)
\end{equation}
and
\begin{equation} \lb{2.8}
\|\mu_{B;\varphi,\psi}\|\leq\|\varphi\|\, \|\psi\|.
\end{equation}

The canonical decomposition for $X$ (see \cite[Sect.\ 1.2]{STI}) says (with $N$ finite or infinite)
\begin{equation} \lb{2.9}
X=\sum_{j=1}^N \mu_j(X) \langle \varphi_j, \dott\rangle \psi_j,
\end{equation}
where $\{\psi_j\}_{j=1}^N$ and $\{\varphi_j\}_{j=1}^N$ are orthonormal sets, $\mu_j>0$, and
\begin{equation} \lb{2.10}
\sum_{j=1}^N \mu_j(X) = \|X\|_{\calI_1}.
\end{equation}

Define
\begin{equation} \lb{2.11}
d\mu_{B,X} =\sum_{j=1}^N \mu_j(X)\, d\mu_{B;\varphi_j,\psi_j}
\end{equation}
which converges by \eqref{2.8} and \eqref{2.10}.
\end{proof}
%%%%%%%%%%%%%%%

%%%%%%%%%%%%%%%
\begin{theorem}\lb{T2.4} Let $A,B$ be bounded operators and $X=(A-B)\in\calI_1$. Let
$\xi (\lambda;A,B)$ be the KrSSF. Let $d\mu_{B,X}$ be given by \eqref{2.6}. Define
\begin{equation} \lb{2.12}
\eta(\lambda;A,B)\equiv\mu_{B,X}((-\infty,\lambda))-\int_{-\infty}^\lambda
\xi(\lambda';A,B)\, d\lambda',  \quad \lambda\in\bbR.
\end{equation}
Then $\eta(\dott;A,B)$ has compact support and for any
$f\in C^\infty(\bbR)$, we have
\begin{equation} \lb{2.13}
\tr\bigg(f(A)-f(B) - \f{d}{d\alpha}\, f(B+\alpha X)\bigg|_{\alpha=0}\bigg)
= \int_{\bbR} f''(\lambda) \eta(\lambda;A,B)\, d\lambda.
\end{equation}
\end{theorem}
%%%%%%%%%%%%%%%

%%%%%%%%%%%%%%%
\begin{remarks}
1. Since
\[
\int_{\bbR} d\mu_{B,X}(\lambda) =\tr(X)=\int_{\bbR} \xi(\lambda;A,B)\, d\lambda,
\]
we can replace $(-\infty,\lambda)$ in both places in \eqref{2.12} by $[\lambda,\infty)$.
This shows that $\eta$ in \eqref{2.12} has compact support.

\smallskip
2. \eqref{2.13} determines $\eta$ uniquely up to an affine term. The condition that
$\eta$ have compact support (as the $\eta$ of \eqref{2.12} does) determines $\eta$
uniquely.
\end{remarks}
%%%%%%%%%%%%%%%%

%%%%%%%%%%%%%%%%
\begin{proof} We first claim that $\left. \f{d}{d\alpha} f(B+\alpha X)\right|_{\alpha =0}$
is trace class and that
\begin{equation} \lb{2.14}
\tr\bigg( \f{d}{d\alpha}\, f(B+\alpha X)\bigg|_{\alpha =0}\bigg) = \tr (Xf'(B)).
\end{equation}
This is immediate for $f$ nice enough (e.g., $f$
such that \eqref{2.5} holds for $n=1$) since
\begin{equation} \lb{2.15}
\tr (e^{i\alpha \beta B}Xe^{i\alpha (1-\beta)B}) = \tr (Xe^{i\alpha B}).
\end{equation}

Thus, by \eqref{2.6},
\begin{align*}
\tr\bigg( \f{d}{d\alpha}\, f(B+\alpha X)\bigg)
&= \int_{\bbR} f'(\lambda)\, d\mu_{B,X}(\lambda) \\
&= -\int_{\bbR} f''(\lambda) [\mu_{B,X} ((-\infty, \lambda))].
\end{align*}

Similarly, by \eqref{1.5},
\begin{align*}
\tr (f(A)-f(B)) &= \int_{\bbR} f'(\lambda) \xi (\lambda;A,B) \, d\lambda  \\
&= -\int_{\bbR} f'' (\lambda) \left(\int_{-\infty}^\lambda \xi (\lambda';A,B)\, d\lambda'\right)
d\lambda.
\qedhere
\end{align*}
\end{proof}
%%%%%%%%%%%%%%

The next critical step will be to prove positivity of $\eta$.

%%%%%%%%%%%%%%%%%%%%%%%%%%%%%%%%%%%%%
\section{Convexity of $\tr(f(A))$} \lb{s3}
%%%%%%%%%%%%%%%%%%%%%%%%%%%%%%%%%%%%%

Positivity of $\eta$ is essentially equivalent to the following result:

%%%%%%%%%%%%%%%%
\begin{theorem}\lb{T3.1} Let $f$ be a convex function on $\bbR$. Then the mapping
\begin{equation} \lb{3.1}
A\mapsto\tr (f(A))
\end{equation}
is a convex function on the $m\times m$ self-adjoint matrices for every $m\in\bbN$.
\end{theorem}
%%%%%%%%%%%%%%%%

%%%%%%%%%%%%%%%%
\begin{remarks} 1. More generally, if $f$ is convex on $(a,b)$, \eqref{3.1} is
convex on matrices $A$ with spectrum in $(a,b)$. In fact, it is easy to see that any
convex function $f$ on $(a,b)$ is a monotone limit on $(a,b)$ of convex functions on
$\bbR$. So this more general result is a consequence of Theorem~\ref{T3.1}.

\smallskip
2. We will discuss the infinite-dimensional situation below.
\end{remarks}
%%%%%%%%%%%%%%

Two special cases of this are widely known and used:
\begin{SL}
\item[(a)] $A\mapsto\tr (e^A)$ is convex.
\item[(b)] $A\mapsto\tr (A\log (A))$ is convex on $A\geq 0$.
\end{SL}
Both of these are rather special. In the first case, one has the stronger $A\mapsto
\log (\tr (e^A))$ is convex and the usual proof of it is via H\"older's
inequality (cf., e.g., \cite[p.\ 19--20]{Is79} or \cite[p.\ 57]{Si93}) which proves the
strong convexity of the $\log(\cdot)$, but does not prove Theorem~\ref{T3.1}. 
In the second case, by Kraus' theorem \cite{Kr,BSh} $A\mapsto A\log (A)$ is 
operator convex. (We also note that $A^r$, $r\in\bbR$, is operator convex for $A>0$ 
if and only if $r\in [-1,0]\cup [1,2]$ (cf.\ \cite[p.\ 147]{Bh97}).) We have found 
Theorem~\ref{T3.1} stated in Alicki--Fannes \cite[Sect.\ 9.1]{AF} and an equivalent 
statement in Ruelle \cite[Sect.\ 2.5]{RueBk} (who attributes it to Klein \cite{Kl31} 
although Klein only has the special case $f(x)= x\log (x)$ and his proof is specific 
to that case; Ruelle's is not). We have also found it in Lieb--Pedersen \cite{LP02} 
whose proof is closer to the one we label ``Third Proof" below. The result is also 
mentioned in von Neumann \cite{vN}, although the proof he gives earlier for a 
special $f$ does not seem to establish the general case.

In any event, even though this result is not hard and is known to some experts, we provide
several proofs because it is not widely known and is central to the theory of KoSSF. We
provide several proofs because they illustrate different aspects of the theorem.

%%%%%%%%%%%%%%%
\begin{proof}[First Proof] This uses eigenvalue perturbation theory. By a limiting argument,
it suffices to prove it for functions $f\in C^\infty(\bbR)$. By approximating derivatives of $f$ by
polynomials, we see that matrix elements, and so the trace of $f(A)$, are $C^\infty$-functions
of $A$. By a limiting argument, we need only show $\lambda\to\tr (A+\lambda X)$ has
a nonnegative second derivative at $\lambda =0$ in case $A$ has distinct eigenvalues.

So by changing basis, we suppose $A$ is diagonal with the eigenvalues $a_1< a_2 < \dots < a_m$.
Let $e_j(\lambda)$ be the eigenvalue of $A+\lambda X$ near $a_j$ for $|\lambda|$
sufficiently small. As is well known \cite[Sect.\ II.2]{Ka80}, \cite[Sect.\ XII.1]{RS4},
\begin{equation} \lb{3.2}
\left. \f{d^2 e_j}{d\lambda^2}\right|_{\lambda =0} =\sum_{\substack{k=1 \\ k\neq j}}^m\,
\f{\abs{X_{k,j}^2}}{a_j-a_k}.
\end{equation}

Clearly,
\begin{align*}
\f{d}{d\lambda}\, [f(e_j(\lambda))] &= f'(e_j(\lambda)) e'_j (\lambda),  \\
\f{d^2}{d\lambda^2}\, f(e_j (\lambda)) &= f'' (e_j(\lambda)) e'_j(\lambda)^2
+ f' (e_j(\lambda)) e''_j (\lambda),
\end{align*}
so
\begin{equation} \lb{3.3}
\left. \f{d^2}{d\lambda^2}\, \tr(f(A(\lambda)))\right|_{\lambda =0} = \boxed{1} + \boxed{2}\, ,
\end{equation}
where
\[
\boxed{1} =\sum_{j=1 }^m f''(a_j) e'_j(0)^2 \geq 0
\]
since $f'' \geq 0$ and, by \eqref{3.2},
\[
\boxed{2} =\sum_{\substack{ j,k=1\\ k\neq j}}^m \abs{X_{k,j}}^2
\biggl[ \f{f'(a_j) -f'(a_k)}{a_j-a_k}\biggr] \geq 0
\]
since for $x <y$,
\[
\f{f'(y)-f'(x)}{y-x} = \f{1}{y-x} \int_x^y f''(u)\, du \geq 0.
\qedhere
\]
\end{proof}
%%%%%%%%%%%%%%%

%%%%%%%%%%%%%%%
\begin{proof}[Second Proof] This one uses a variational principle. We consider first the case
\begin{equation} \lb{3.4}
f^+(x)=x_+ \equiv\begin{cases}
x, & x\geq 0, \\
0, & x <0.
\end{cases}
\end{equation}
We claim first that
\begin{equation} \lb{3.5}
\tr(f^+(A)) = \max\{\tr(AB) \, | \, \|B\|\leq 1,\, B\geq 0\},
\end{equation}
where $\|\cdot\|$ is the matrix norm on $\bbC^m$ with the Euclidean norm. For in an orthonormal
basis where $A$ is a diagonal matrix,
\[
\tr(AB)=\sum_{j=1}^m a_j b_{j,j}\leq \sum_{j=1}^m (a_j)_+ =\tr(f^+(A))
\]
if $0\leq b_{jj}\leq 1$. On the other hand, if $B$ is the diagonal matrix with
\[
b_{jj}=\begin{cases}
1, & a_j >0,  \\
0, & a_j\leq 0,
\end{cases}
\]
then $B\geq 0$, $\|B\|\leq 1$, and $\tr(AB)=\tr(f^+(A))$. This proves \eqref{3.5}.

Convexity is immediate for $f^+$ given by \eqref{3.4} once we have \eqref{3.5}, since
maxima of linear functionals are convex. Obviously, since $(x-\lambda)_+$ is just
a translate of $x_+$, we get convexity for any function of
\[
\int_{\lambda_0}^\infty (x-\lambda)_+ \, d\mu (\lambda)
\]
for any Borel measure $\mu$ on $(\lambda_0, \infty)$. But every convex function
$f$ with $f\equiv 0$ for $x\leq \lambda_0$ has this form.

Adding $ax+b$ to this, we get the result for any convex function $f$ with $f''(x)=0$ for
$x\leq \lambda_0$. Taking $\lambda_0\to -\infty$, we get the result for general convex
functions $f$.
\end{proof}
%%%%%%%%%%%%%%%

%%%%%%%%%%%%%%%
\begin{proof}[Third Proof {\rm{(}}M.~B.~Ruskai, private communication{\rm{)}}]
If $f$ is any convex function, $C$ a self-adjoint $m\times m$ matrix with
\begin{equation} \lb{3.6}
Ce_j =\lambda_je_j,
\end{equation}
and $v\in\bbC^m$ a unit vector, then
\begin{align}
\langle v, f(C)v\rangle &= \sum_{j=1}^m \abs{\langle v, e_j\rangle}^2 f(\lambda_j) \notag \\
&\geq f \biggl( \sum_{j=1}^m \lambda_j \abs{\langle v,e_j\rangle}^2\biggr) \lb{3.7} \\
&= f (\langle v, Cv\rangle),   \lb{3.7a}
\end{align}
where \eqref{3.7} employs Jensen's inequality.

Now suppose
\[
C=\theta A + (1-\theta) B, \quad \theta \in [0,1].
\]
Then,
\begin{align}
\tr (f(C)) &=\sum_{j=1}^m f(\langle e_j, Ce_j\rangle ) \notag \\
& = \sum_{j=1}^m f(\theta \langle e_j, Ae_j\rangle + (1-\theta)
\langle e_j, Be_j\rangle) \notag \\
&\leq \sum_{j=1}^m [\theta f(\langle e_j, Ae_j\rangle) + (1-\theta)
f(\langle e_j, Be_j\rangle)] \lb{3.8} \\
&\leq \theta \sum_{j=1}^m \langle e_j, f(A)e_j\rangle + (1-\theta) \sum_{j=1}^m
\langle e_j, f(B)e_j\rangle \lb{3.9} \\
&= \theta \tr(f(A))+(1-\theta) \tr(f(B)),
\end{align}
proving convexity. In the above, \eqref{3.8} is direct convexity of $f$ and
\eqref{3.9} is \eqref{3.7a} for $v=e_j$ and $C=A$ or $B$.
\end{proof}
%%%%%%%%%%%%%%%

%%%%%%%%%%%%%%%
\begin{corollary} \lb{C3.2} If $f\in C^1(\bbR)$ is convex and $B$ and $X$ are
$m\times m$ self-adjoint matrices, $m\in\bbN$, then
\begin{equation} \lb{3.10}
\tr\bigg(f(B+X)-f(B) - \f{d}{d\alpha}\, f(B+\alpha X)\bigg|_{\alpha=0}\bigg)
\geq 0.
\end{equation}
\end{corollary}
%%%%%%%%%%%%%%%

%%%%%%%%%%%%%%%
\begin{remarks} 1. It is not hard to see that \eqref{3.10} is equivalent to
Theorem~\ref{T3.1}.

\smallskip
2. It is in this form that the result appears in Ruelle \cite[Sect.\ 2.5]{RueBk},
and for the case $f(x)=x\log (x)$, $x>0$, in Klein \cite{Kl31}.
\end{remarks}
%%%%%%%%%%%%%%%

%%%%%%%%%%%%%%%
\begin{proof} If $g\in C^1(\bbR)$ is a convex function,
\begin{equation} \lb{3.11}
g(x+y) -g(x)-g'(x)y \geq 0,
\end{equation}
since convexity says that $g$ lies above the tangent line at any point. \eqref{3.10}
is \eqref{3.11} for $g(\alpha)=\tr (f(B+\alpha X))$, $x=0$, $y=1$.
\end{proof}
%%%%%%%%%%%%%%%

%%%%%%%%%%%%%%%
\begin{corollary}\lb{C3.3} For finite-dimensional matrices $A$ and $B$, the KoSSF,
$\eta(\dott;A,B)$, satisfies $\eta(\lambda;A,B) \geq 0$ for a.e.\ $\lambda\in\bbR$.
\end{corollary}
%%%%%%%%%%%%%%%
\begin{proof}
Let $h:\bbR\to[0,\infty)$ be a measurable function bounded and supported on an interval $(a,b)$
with $\sigma(A)\cup\sigma(B)\subset(a,b)$ (so, by \eqref{2.12}, $\eta$ is supported on
$(a,b)$). Let $f$ be the unique convex function with $f=0$ near $-\infty$ and $f''=h$.
By \eqref{3.10} and \eqref{2.13},
\begin{equation} \lb{3.12}
0 \leq \int_{\bbR} h(\lambda) \eta(\lambda;A,B)\, d\lambda.
\end{equation}
Since $h$ is arbitrary, $\eta\geq 0$ a.e.
\end{proof}
%%%%%%%%%%%%%%

%%%%%%%%%%%%%%
\begin{theorem}\lb{T3.4} For any finite self-adjoint matrices $A,B$ {\rm{(}}of the same
size{\rm{)}},
\begin{equation} \lb{3.13}
\int_{\bbR} \, \abs{\eta(\lambda;A,B)}\, d\lambda = \tfrac12\, \|A-B\|_{\calI_2}^2.
\end{equation}
\end{theorem}
%%%%%%%%%%%%%%

%%%%%%%%%%%%%%
\begin{remarks} 1. It is remarkable that we always have equality in \eqref{3.13}.
The analog for the KrSSF is
\begin{equation} \lb{3.14}
\int_{\bbR} \, \abs{\xi(\lambda;A,B)}\, d\lambda \leq \|A-B\|_{\calI_1},
\end{equation}
where equality, in general, holds if $A-B$ is either positive or negative.

\smallskip
2. \eqref{3.13} emphasizes again the lack of a chain rule for $\eta$; $\eta$ is nonlinear
in $(A-B)$.
\end{remarks}
%%%%%%%%%%%%%%

%%%%%%%%%%%%%%
\begin{proof} Take $f(x)=\f12 x^2$ such that $f''(x)=1$ and
\begin{align*}
&f(B+X)-f(B) -\left. \f{d}{d\alpha}\, f(B+ \alpha X)\right|_{\alpha =0}  \\
&\quad = \tfrac12\, \big[(B+X)^2 -B^2 -XB-BX\big] =\tfrac12\, X^2.
\end{align*}
Since $\eta\geq 0$, $\int_{\bbR} f''(\lambda)\eta(\lambda;A,B)\, d\lambda =
\int_{\bbR} \abs{\eta(\lambda;A,B)}\, d\lambda$ and \eqref{3.13} holds.
\end{proof}
%%%%%%%%%%%%%

In Section~\ref{s5} we take limits from the finite-dimensional situation, but one can
easily extend Theorem~\ref{T3.1} in two ways and from there directly prove $\eta\geq 0$
and \eqref{3.13} in case $(A-B)\in\calI_1$. Without proof, we state the extensions (the
results are simple limiting arguments from finite dimensions):

%%%%%%%%%%%%%
\begin{theorem}\lb{T3.5} If $f$ is convex on $\bbR$ and $f(0)=0$, then $f(A)$ is trace
class for any self-adjoint trace class operator $A$, and for such $A$'s, the mapping
$A\mapsto \tr(f(A))$ is convex.
\end{theorem}
%%%%%%%%%%%%%

In this context we note that convex functions are Lipschitz continuous. (For this and additional
regularity results of convex functions, see, e.g., \cite[p.\ 145--146]{Bh97}.)

%%%%%%%%%%%%%
\begin{theorem}\lb{T3.6} For any convex function $f\in C^\infty(\bbR)$, any bounded
self-adjoint operator $B$, and any self-adjoint operator $X\in\calI_1$,
$$
[f(B+X)-f(B)]\in\calI_1
$$
and the mapping $X\mapsto\tr (f(B+X)-f(B))$ is convex.
\end{theorem}
%%%%%%%%%%%%%

Convexity of maps of the type $s\mapsto \tr(f(B+X(s))-f(B))$, $s\in(s_1,s_2)$,
for convex $f$ and certain classes of $X(\cdot)\in\calI_1$ was also studied in \cite{GMM00}.

%%%%%%%%%%%%%%%%%%%%%%%%%%%%%%%%%%%%%
\section{L\"owner's Formula and the Finite-Dimensional Birman--Solomyak Bound}
\lb{s4}
%%%%%%%%%%%%%%%%%%%%%%%%%%%%%%%%%%%%%

The final element needed to construct the KoSSF is the following lovely theorem of
Birman--Solomyak \cite{BS2} (see also \cite{BS4}):

%%%%%%%%%%%%%%
\begin{theorem}\lb{T4.1} Let $A,B$ be bounded self-adjoint operators with
$(A-B)$ Hilbert--Schmidt. Let $f$ be a function defined on an interval
$[a,b]\supset \sigma(A)\cup\sigma(B)$. Suppose $f$ is uniformly Lipschitz, that is,
\begin{equation} \lb{4.1}
\|f\|_L =\sup_{\substack{x,y\in [a,b]\\x\neq y}}\, \f{\abs{f(x)-f(y)}}{\abs{x-y}} <\infty.
\end{equation}
Then $[f(A)-f(B)]$ is also Hilbert--Schmidt and
\begin{equation} \lb{4.2}
\|f(A)-f(B)\|_{\calI_2} \leq \|f\|_L \|A-B\|_{\calI_2}.
\end{equation}
\end{theorem}
%%%%%%%%%%%%%%

The proof in \cite{BS2} depends on the deep machinery of double Stieltjes operator integrals.
Our two points in this section are:
\begin{SL}
\item[(1)] The inequality for finite matrices is quite elementary and, by limits, extends
to \eqref{4.2}.
\item[(2)] The key to our proof, a kind of ``Double Stieltjes Operator Integral for Dummies,"
goes back to L\"owner \cite{Loe1} in 1934 whose contributions to this theme seem not to
have been appreciated in the literature on double Stieltjes operator integrals.
\end{SL}

Given two finite $m\times m$ self-adjoint matrices $A,B$ with respective eigenvectors
$\{\varphi_j\}_{j=1}^m$ and $\{\psi_j\}_{j=1}^m$ and eigenvalues $\{x_j\}_{j=1}^m$
and $\{y_j\}_{j=1}^m$ such that
\begin{equation} \lb{4.3}
A\varphi_j =x_j\varphi_j,  \qquad
B\psi_j = y_j \psi_j,
\end{equation}
we introduce the (modified) L\"owner matrix of a function $f$ by
\begin{equation} \lb{4.4}
L_{k,\ell} = \begin{cases} \f{f(y_k)-f(x_\ell)}{y_k-x_\ell}, & y_k \neq x_\ell, \\
0, & y_k = x_\ell, \end{cases} \quad 1\leq k,\ell\leq m.
\end{equation}
(L\"owner \cite{Loe1} originally supposed $y_k \neq x_\ell$ for all $1\leq k,\ell \leq m$.)
Clearly, if $f$ is Lipschitz,
\begin{equation} \lb{4.5}
\sup_{1\leq k,\ell\leq m}\, \abs{L_{k,\ell}} \leq \|f\|_L.
\end{equation}

L\"owner noted that since
\begin{equation} \lb{4.6}
f(A)\varphi_j=f(x_j)\varphi_j, \qquad
f(B)\psi_j =f(y_j) \psi_j,
\end{equation}
we have L\"owner's formula:
\begin{equation} \lb{4.7}
\langle\psi_k,[f(B)-f(A)]\varphi_\ell\rangle = L_{k\ell}
\langle\psi_k, (B-A)\varphi_\ell\rangle,
\end{equation}
and this holds even if $y_k=x_\ell$ (since then both matrix elements vanish).
This is the ``baby'' version of the double Stieltjes operator integral formula
$$
f(B)-f(A)=\int_{\sigma(A)}\int_{\sigma(B)} \frac{f(y)-f(x)}{y - x} \,
dE_B(x)(B-A)dE_A(y)
$$
due to Birman and Solomyak  \cite{BS2, BS4}. Here the integration is with respect to
the spectral measures of $A$ and $B$.

L\"owner's formula immediately implies:

%%%%%%%%%%%%%
\begin{proposition}\lb{P4.2} \eqref{4.2} holds for finite self-adjoint matrices.
\end{proposition}
%%%%%%%%%%%%%
\begin{proof} Hilbert--Schmidt norms can be computed in any basis, even two
different ones, that is,
\[
\|C\|_{\calI_2}^2 =\sum_{\ell=1}^m \, \|C\varphi_\ell\|^2 =\sum_{k,\ell=1}^m \,
\abs{\langle\psi_k, C\varphi_\ell\rangle}^2.
\]
Thus,
\begin{alignat*}{2}
\|f(A)-f(B)\|_{\calI_2}^2 &= \sum_{k,\ell=1}^m\, \abs{\langle\psi_k, (f(B)-f(A))
  \varphi_\ell\rangle}^2 \notag \\
&=\sum_{k,\ell=1}^m L_{k\ell}^2 \abs{\langle\psi_k, (B-A)\varphi_\ell\rangle}^2
\qquad && \text{by \eqref{4.7}} \\
&\leq \|f\|_L^2 \sum_{k,\ell=1}^m\, \abs{\langle\psi_k, (B-A)\varphi_\ell\rangle}^2
\qquad && \text{by \eqref{4.5}} \\
&= \|f\|_L^2 \|A-B\|_{\calI_2}^2.
\end{alignat*}
\end{proof}
%%%%%%%%%%%%%

%%%%%%%%%%%%%
\begin{proof}[Proof of Theorem~\ref{T4.1}] Let $\{\zeta_j\}_{j=1}^\infty$ be an orthonormal
basis for $\calH$ and $P_N$ the orthogonal projections onto the linear span of
$\{\zeta_j\}_{j=1}^N$.\ For any $A$ and $B$ and Lipschitz $f$, by Proposition~\ref{P4.2},
\begin{align}
\|f(P_N BP_N) -f(P_N AP_N)\|_{\calI_2}
&\leq \|f\|_L \|P_N(B-A) P_N\|_{\calI_2} \lb{4.8} \\
&\leq \|f\|_L \|B-A\|_{\calI_2}     \lb{4.9}
\end{align}
if $B-A$ is Hilbert--Schmidt, since
\begin{align}
\|P_N (B-A) P_N\|_{\calI_2}^2 &=\sum_{j=1}^n \, \|P_N (B-A)\zeta_j\|^2 \leq
\sum_{j=1}^n \, \|(B-A)\zeta_j\|^2   \\
& \leq \|B-A\|_{\calI_2}^2.
\end{align}

Thus, for any $k\in\bbN$,
\begin{equation} \lb{4.10}
\sum_{j=1}^k\, \|[f(P_N BP_N) -f(P_N AP_N)]\zeta_j\|^2 \leq \|f\|_L^2
\|B-A\|_{\calI_2}^2.
\end{equation}

As $P_N BP_N \underset{N\to\infty}{\longrightarrow} B$ strongly, one infers by continuity
of the functional calculus that $f(P_N BP_N) \underset{N\to\infty}{\longrightarrow} f(B)$
strongly. Since the sum in \eqref{4.10} is finite, one concludes that
\[
\sum_{j=1}^k\, \|(f(B)-f(A))\zeta_j\|^2 \leq \|f\|_L^2 \|B-A\|_{\calI_2}^2.
\]
Taking $k\to\infty$, we see that $[f(B)-f(A)] \in\calI_2$ and that \eqref{4.2} holds.
\end{proof}
%%%%%%%%%%%%

%%%%%%%%%%%%%%%%%%%%%%%%%%%%%%%%%%%%
\section{General Construction of the KoSSF $\eta(\dott;A,B)$} \lb{s5}
%%%%%%%%%%%%%%%%%%%%%%%%%%%%%%%%%%%%

The general construction and proof of properties of $\eta$ depends first on an approximation
of trace class operators by finite rank ones and then on an approximation of Hilbert--Schmidt
operators by trace class operators. In this section, we mostly follow the approach
of \cite[Lemma 3.3]{Kop}.

%%%%%%%%%%%%
\begin{theorem}\lb{T5.1} Let $B_n,B$, $n\in\bbN$, be uniformly bounded self-adjoint operators
such that $B_n\underset{n\to\infty}{\longrightarrow}B$ strongly. Let $X_n,X$, $n\in\bbN$, be
a sequence of self-adjoint trace class operators such that $\|X-X_n\|_{\calI_1}
\underset{n\to\infty}{\longrightarrow} 0$. Then for any continuous function, $g$,
of compact support, we conclude that
\begin{equation} \lb{5.1}
\int_{\bbR} g(\lambda) \eta(\lambda;B_n +X_n,B_n)\, d\lambda \underset{n\to\infty}{\longrightarrow}
\int_{\bbR} g(\lambda) \eta (\lambda; B+X,B)\, d\lambda.
\end{equation}
In particular, $\eta (\dott;A,B)\geq 0$ a.e.\ on $\bbR$ if $(A-B)\in\calI_1$ and,
in that case, \eqref{3.13} holds.
\end{theorem}
%%%%%%%%%%%%
\begin{proof}  By Theorem~\ref{TA.6} and
\[
\int_{\bbR} g(\lambda) \biggl(\int_{-\infty}^\lambda \xi (\lambda'; A,B)\, d\lambda'\biggr)\, 
d\lambda = \int_{-\infty}^\infty \xi (\lambda'; A,B)\biggl( \int_{\lambda'}^\infty g(\lambda)\, 
d\lambda\biggr) \, d\lambda'
\]
we get convergence of the second term in \eqref{2.12}.
By \eqref{2.6},
\[
d\mu_{B_n,X_n} \underset{n\to\infty}{\longrightarrow} d\mu_{B,X}
\]
weakly by the strong continuity of the functional calculus since
\begin{align*}
& \abs{\tr (X_n f(B_n))-\tr(Xf(B))}   \\
& \quad \leq \abs{\tr(X[f(B_n)-f(B)])} +
\|f\|_\infty \|X-X_n\|_{\calI_1} \underset{n\to\infty}{\longrightarrow} 0,
\end{align*}
as $f$ is continuous.

Since weak limits of positive measures are positive, the positivity follows from positivity
in the finite-dimensional case taking $B_n=P_n BP_n$ and $X_n=P_n XP_n$ for finite-dimensional
$P_n$ converging strongly to $I$, the identity operator.

Once we have positivity, we obtain \eqref{3.13} directly by following the proof of
Theorem~\ref{T3.4}.
\end{proof}
%%%%%%%%%%%%%%

%%%%%%%%%%%%%%
\begin{theorem}\lb{T5.2} Let $A,B,C$ be bounded self-adjoint operators such that
$(A-C)\in\calI_1$ and $(B-C)\in\calI_1$. Then
\begin{equation} \lb{5.2}
\int_{\bbR} \abs{\eta (\lambda;A,C)-\eta(\lambda;B,C)}\, d\lambda \leq \|A-B\|_{\calI_2}
\big[\tfrac12\, \|A-B\|_{\calI_2} + \|B - C\|_{\calI_2}\big].
\end{equation}
\end{theorem}
%%%%%%%%%%%%%%
\begin{proof} We begin with \eqref{1.7a} which follows from the fact that \eqref{2.13}
holds when $(A-B)\in\calI_1$. Here \eqref{1.8} holds for nice functions $g$, say,
$g\in C^\infty(\bbR)$. Thus,
\begin{equation}
\text{LHS of \eqref{5.2}} \leq \int_{\bbR} \abs{\eta (\lambda; A,B)}\, d\lambda
+ \int_{\bbR} \abs{\delta\eta (\lambda;A,B,C)}\, d\lambda.   \lb{5.3}
\end{equation}
By \eqref{3.13},
\begin{equation}
\text{First term on RHS of \eqref{5.3}} \leq
\tfrac12\, \|A-B\|_{\calI_2} \, \|A-B\|_{\calI_2}.   \lb{5.4}
\end{equation}

As for $\delta\eta$, by \eqref{1.8},
\begin{align*}
\biggl| \int_{\bbR} g'(\lambda) \delta\eta (\lambda)\,d\lambda \biggr|
&\leq \|A-B\|_{\calI_2} \, \|g(B)-g(C)\|_{\calI_2} \\
&\leq \|g'\|_\infty \|A-B\|_{\calI_2} \, \|B-C\|_{\calI_2}
\end{align*}
by Theorem~\ref{T4.1}. Since $\delta\eta\in L^1(\bbR)$ and the bounded
$C^\infty(\bbR)$-functions are $\|\cdot\|_\infty$-dense in the bounded
continuous functions, and for $h\in L^1(\bbR)$,
\[
\|h\|_1 =\sup_{\substack{f\in C(\bbR) \\ \|f\|_\infty =1}} \,
\biggl| \int_{\bbR} f(x) h(x) \, dx \biggr|,
\]
we conclude
\begin{equation} \lb{5.5}
\|\delta\eta\|_1 \leq \|A-B\|_{\calI_2} \, \|B - C\|_{\calI_2}.
\end{equation}
Relations \eqref{5.3}--\eqref{5.5} imply \eqref{5.2}.
\end{proof}
%%%%%%%%%%%%

Here is the main theorem on the existence of the KoSSF:

%%%%%%%%%%%%
\begin{theorem}\lb{T5.3} Let $A,B$ be two bounded self-adjoint operators with
$(B-A) \in\calI_2$. Then there exists a unique $L^1(\bbR;d\lambda)$-function
$\eta (\dott; A,B)$ supported on $(-\max (\|A\|, \|B\|), \max (\|A\|,\|B\|))$
such that for any $g\in C^\infty(\bbR)$,
\begin{equation} \lb{5.6}
\bigg(g(A)-g(B)- \f{d}{d\alpha}\, g(B+\alpha (A-B))\bigg|_{\alpha =0}\bigg)
\in\calI_1
\end{equation}
and
\begin{equation} \lb{5.7}
\tr\bigg(g(A)-g(B)- \f{d}{d\alpha}\, g(B+\alpha (A-B))\bigg|_{\alpha =0}\bigg)
 =\int_{\bbR} g''(\lambda) \eta(\lambda;A,B)\, d\lambda.
\end{equation}
Moreover,
\begin{gather}
 \eta(\dott;A,B) \geq 0 \, \text{ a.e.\ on $\bbR$,}   \\
 \int_{\bbR} \abs{\eta (\lambda;A,B)}\, d\lambda = \tfrac12\, \|A-B\|_{\calI_2}^2,  \lb{5.8}
\end{gather}
and for any bounded self-adjoint operators $A,B,C$ with $(A-C) \in\calI_2$ and
$(B-C) \in\calI_2$,
\begin{equation} \lb{5.9}
\int_{\bbR} \abs{\eta(\lambda;A,C) - \eta (\lambda;B,C)}\, d\lambda \leq \|A-B\|_{\calI_2}
\big[\tfrac12\, \|A-B\|_{\calI_2} + \|B-C\|_{\calI_2}\big].
\end{equation}
\end{theorem}
%%%%%%%%%%%%%

%%%%%%%%%%%%%
\begin{remark} For a sharp condition on the class of functions $\eta$ for which
Koplienko's trace formula holds, we refer to Peller \cite{Pe05}.
\end{remark}
%%%%%%%%%%%%%

%%%%%%%%%%%%%
\begin{proof} Let $X=A-B$ and pick $X_n\in\calI_1$, $n\in\bbN$, such that
$\|X_n-X\|_{\calI_2} \underset{n\to\infty}{\longrightarrow} 0$. By \eqref{5.9},
which we have proven for
\[
\eta (\lambda;B+X_n,B)-\eta (\lambda;B+X_m,B),
\]
we see that $\eta (\dott;B+X_n,B)$ is Cauchy in $L^1(\bbR;d\lambda)$ and so converges
a.e.\ to what we will define as
$\eta (\dott;A,B)$. \eqref{5.7} holds by taking limits; $\eta\geq 0$ as a limit of positive
functions $\eta$. \eqref{5.8} and \eqref{5.9} hold by taking limits.
\end{proof}
%%%%%%%%%%%%

%%%%%%%%%%%%
\begin{remark}
Here is an alternative method of proving the estimate (5.2), bypassing Theorem 5.1:
In the same way as in Corollary 3.3, one can deduce positivity
of $\eta(\dott;A,B)$ for $(A-B)\in\calI_1$ from Theorem 3.6.
The only place where Theorem 5.1 is used in the proof of Theorem 5.2 is
in the estimate (5.4). This estimate (see Theorem 3.4) follows directly from
the positivity of $\eta$ and the trace  formula (2.15).
\end{remark}
%%%%%%%%%%%%

%%%%%%%%%%%%%%%%%%%%%%%%%%%%%%%%%%%%%%
\section{What Functions $\eta$ are Possible?}  \lb{s5a}
%%%%%%%%%%%%%%%%%%%%%%%%%%%%%%%%%%%%%%

We introduce the classes of functions
\begin{align*}
\eta(\calI_2)&=\{\eta(\dott;A,B)\,|\, A, B \text{ bounded and self-adjoint,} \, (A-B) \in\calI_2\},\\
\eta(\calI_1)&=\{\eta(\dott;A,B)\,|\, A, B \text{ bounded and self-adjoint,} \, (A-B) \in\calI_1\}.
\end{align*}
In this section we would like to raise the question of the description of
the classes $\eta(\calI_2)$ and $\eta(\calI_1)$.

Since, for now, we are considering $A,B$ bounded, $\eta$ has compact support.

First we discuss the class $\eta(\calI_2)$. By Theorem~\ref{T5.3},
all functions of this class are nonnegative and Lebesgue integrable.
It would be interesting to see if the class $\eta(\calI_2)$ contains all
nonnegative Lebesgue integrable functions.

As a step towards answering this question, we give the following elementary
result:

%%%%%%%%%%%%%%
\begin{theorem} \label{t6.1}
The class  $\eta(\calI_2)$ contains all nonnegative Riemann integrable functions
of compact support.
\end{theorem}
%%%%%%%%%%%%

%%%%%%%%%%%%%%%
\begin{remark} A contemporary account of the theory of Riemann integrable functions
can be found, for instance, in Stein--Shakarchi \cite{SS}.
\end{remark}
%%%%%%%%%%%%%%

%%%%%%%%%%%%%%
\begin{proof}
First we consider a simple example. Let $a\in \bbR$ and $\varepsilon>0$; consider the
operators in $\bbC^2$ given by the diagonal $2\times 2$ matrices
$B=\diag(a-\varepsilon,a+\varepsilon)$ and $A=\diag(a+\varepsilon,a-\varepsilon)$.
Then the KoSSF for this pair is given by (cf.\ \eqref{2.12})
$\eta(\lambda)=2\varepsilon \chi_{(a-\epsilon, a+\epsilon)}(\lambda)$.
We  note that
$$
\int_{\bbR}\eta(\lambda) \, d\lambda =4\varepsilon^2=\tfrac12 \tr ((A-B)^2),
$$
in agreement with \eqref{5.8}.

Next, suppose that $0\leq \eta\in L^1(\bbR;d\lambda)$ is represented by the
$L^1(\bbR;d\lambda)$-convergent series
\begin{equation}
\eta(\lambda)=\sum_{n=1}^\infty
\abs{I_n}\chi_{I_n}(\lambda),
\label{a3}
\end{equation}
where $I_n\subset \bbR$ are (not necessarily disjoint) finite intervals and
$\abs{I_n}$ is the length of $I_n$.
Denote by $a_n$ the midpoint of $I_n$ and let $\varepsilon_n=\frac12\abs{I_n}$.
We introduce
$$B=\oplus_{n=1}^\infty\diag(a_n-\varepsilon_n,a_n+\varepsilon_n) \, \text{ and } \,
A=\oplus_{n=1}^\infty\diag(a_n+\varepsilon_n,a_n-\varepsilon_n)
$$
in the Hilbert space   $\oplus_{n=1}^\infty \bbC^2$.
Note that the $L^1(\bbR;d\lambda)$-convergence of the series \eqref{a3} is equivalent to the
condition $\sum_{n=1}^\infty \varepsilon_n^2<\infty$ and so
 $A-B=\oplus_{n=1}^\infty\diag(2\varepsilon_n,-2\varepsilon_n)$
is a Hilbert--Schmidt operator.
It is clear that the
KoSSF for the pair $A, B$  coincides with $\eta$.

Thus, it suffices to prove that any Riemann integrable function
$0\leq \eta\in L^1(\bbR;d\lambda)$
can be represented as an $L^1(\bbR;d\lambda)$-convergent series \eqref{a3}.

Let $0\leq \eta\in L^1(\bbR;d\lambda)$ be Riemann integrable.
According to the definition of the Riemann integral,
there exists a finite set of disjoint open squares
$Q_n$, $n\in \{1,\dots,M\}$, which fit under the graph of $\eta$ and
$$
\sum_{n=1}^M \text{area}(Q_n)\geq\tfrac12 \int_{\bbR} \eta(\lambda) \, d\lambda.
$$
In other words, there exists a finite set of (not necessarily disjoint)
open intervals $I_n\subset \bbR$, $n\in \{1,\dots,N\}$,
such that
\begin{align*}
& \sum_{n=1}^N \abs{ I_n} \chi_{I_n}(\lambda)\leq \eta(\lambda),
\quad \lambda\in\bbR,
\\
& \int_{\bbR} \bigg(\sum_{n=1}^N \abs{ I_n} \chi_{I_n}(\lambda) \bigg)  d\lambda
= \sum_{n=1}^N \abs{ I_n}^2\geq \tfrac12 \int_{\bbR} \eta(\lambda) \, d\lambda.
\end{align*}
Thus, we can represent $\eta$ as
\begin{align*}
& \eta(\lambda)
= \sum_{n=1}^{N_1} \abs{I_n}\chi_{I_n}(\lambda)
+\eta_1(\lambda),  \\
&  \eta_1(\lambda)\geq0,  \quad
\int_{\bbR} \eta_1(\lambda) \, d\lambda \leq
\tfrac12 \int_{\bbR} \eta(\lambda) \, d\lambda,
\end{align*}
and the sum is taken over a finite set of indices $n\in \{1,\dots,N_1\}$.
Iterating this procedure, we see that for any $m\in\bbN$ we can represent
$\eta$ as
\begin{align*}
& \eta(\lambda)= \sum_{n=1}^{N_m} \abs{I_n}\chi_{I_n}(\lambda)
+\eta_m(\lambda),  \\
& \eta_m(\lambda)\geq0,
\quad
\int_{\bbR} \eta_m(\lambda) \, d\lambda
\leq
2^{-m}
\int_{\bbR} \eta(\lambda) \, d\lambda,
\end{align*}
where $\varepsilon_n>0$, $a_n\in\bbR$, and the sum is taken over a finite set of indices $n$.
Taking $m\to\infty$, it follows that $\eta$ can be represented as an
$L^1(\bbR;d\lambda)$-convergent series \eqref{a3}.
\end{proof}
%%%%%%%%%%%%%%

Regarding the class $\eta(\calI_1)$, we note only that every function of this class is
of bounded variation. This follows from \eqref{2.12}, since both terms on the right-hand
side of \eqref{2.12} are of bounded variation. We also note that it follows from the
proof of Theorem \ref{t6.1} that the class $\eta(\calI_1)$ contains all functions of the type
$$
\eta(\lambda)=\sum_{n=1}^\infty
\abs{I_n}\chi_{I_n}(\lambda),
\quad \sum_{n=1}^\infty \abs{I_n}<\infty.
$$

%%%%%%%%%%%%%%%%%%%%%%%%%%%%%%%%%%%
\section{Modified Determinants and the KoSSF}  \lb{new-s7}
%%%%%%%%%%%%%%%%%%%%%%%%%%%%%%%%%%%

In this section, as a preliminary to the next, we want to use our viewpoint to prove a
formula for modified perturbation determinants in terms of the KoSSF originally derived
by Koplienko \cite{Kop}. We recall that one of Krein's motivating formulas for the
KrSSF is (see \eqref{A.31}):
\begin{equation}\lb{new7.1}
\det((A-z)(B-z)^{-1}) = \exp\biggl( \int_\bbR (\lambda-z)^{-1} \xi (\lambda)\,
d\lambda\biggr), \quad z\in\bbC\backslash\bbR.
\end{equation}
Here $\det(\cdot)$ is the Fredholm determinant defined on $I+\calI_1$ (since
$A-B=X\in\calI_1$ implies $(A-z)(B-z)^{-1}-I =X(B-z)^{-1}\in\calI_1$); see \cite{GK,STI}.

We recall that for $C\in\calI_1$, one can define $\det_2(\cdot)$ by
\begin{equation}\lb{new7.2}
\det_2 (I+C)=\det (I+C) e^{-\tr(C)}
\end{equation}
and that $C\mapsto\det_2(I+C)$ extends uniquely and continuously to $\calI_2$, the Hilbert--Schmidt
operators, although the right-hand side of \eqref{new7.2} no longer makes sense (see
\cite[Ch.\ IV]{GK}, \cite[Ch.\ 9]{STI}). Our goal in this section is to prove the
following formula first derived by Koplienko \cite{Kop}:

%%%%%%%%%%%%%%%%%
\begin{theorem}\lb{newT7.1} Let $B$ and $X$ be bounded self-adjoint operators and
$X\in\calI_2$. Let $A=B+X$. Then for any $z\in\bbC\backslash\bbR$,
$(A-z)(B-z)^{-1}\in I +\calI_2$ and
\begin{equation}\lb{new7.3}
\det_2 ((A-z)(B-z)^{-1}) = \exp\biggl( -\int_{\bbR}
\f{\eta(\lambda;A,B)}{(\lambda-z)^2}\, d\lambda\biggr).
\end{equation}
\end{theorem}
%%%%%%%%%%%%%%%%%
\begin{proof} It suffices to prove \eqref{new7.3} for $X\in\calI_1$ since both sides are
continuous in $\calI_2$ norm and $\calI_1$ is dense in $\calI_2$. Continuity of the left-hand
side follows from Theorem~9.2(c) of \cite{STI} and of the right-hand side by Theorem~\ref{T5.2}
above.

When $X\in\calI_1$, we can use \eqref{new7.2}. Let
\begin{equation}\lb{new7.4}
g_1(\lambda) =\mu_{B,X}((-\infty, \lambda)), \quad
g_2(\lambda) = \int_{-\infty}^\lambda \xi(\lambda'; A,B)\, d\lambda'.
\end{equation}

By an integration by parts argument (using $g'_2=\xi$),
\begin{equation}\lb{new7.5}
\int_{\bbR} \f{\xi(\lambda)}{\lambda-z}\, d\lambda
= \int_{\bbR} \f{g_2(\lambda)}{(\lambda-z)^2}\, d\lambda.
\end{equation}
By an integration by parts in a Stieltjes integral and by \eqref{2.6},
\begin{align}
\tr (X(B -z)^{-1}) &= \int_{\bbR} \f{1}{\lambda-z}\, d\mu_{B,X}(\lambda) \lb{new7.6} \\
&= \int_{\bbR} g_1(\lambda)\, \f{1}{(\lambda-z)^2}\, d\lambda.     \lb{new7.7}
\end{align}

Thus, by \eqref{new7.1} and \eqref{new7.2},
\[
\det_2 (1+X(B-z)^{-1})
=\exp\biggl( -\int_{\bbR} (g_1(\lambda) -g_2(\lambda))(\lambda-z)^{-2}\, d\lambda\biggr),
\]
which, given \eqref{2.12} is \eqref{new7.3}.
\end{proof}
%%%%%%%%%%%%%%%%%

%%%%%%%%%%%%%%%%%%%%%%%%%%%%%%%%%%%%%%%%%%%%%%%%%%%%%%%%%%%%%%%%%%%%%%
\section{On Boundary Values of Modified Perturbation Determinants
$\det_2((A-z)(B-z)^{-1})$}  \lb{s6}
%%%%%%%%%%%%%%%%%%%%%%%%%%%%%%%%%%%%%%%%%%%%%%%%%%%%%%%%%%%%%%%%%%%%%%

By \eqref{new7.1}, if $(A-B)\in\calI_1$, $\det((A-z)(B-z)^{-1})$ has a limit as
$z\to \lambda+i0$ for a.e.\ $\lambda \in\bbR$ since $\int_{\bbR} (\nu-z)^{-1} 
\xi(\nu)\, d\nu$ is a difference of Herglotz functions. In this section, we will 
consider nontangential boundary values to the real axis of modified perturbation 
determinants
$$
{\det}_2((A-z)(B-z)^{-1}), \quad z\in\bbC_+,
$$
where $X=(A-B)\in\calI_2$. Unlike the trace class, we will see nontangential
boundary values may not exist a.e.\ on $\bbR$.

For notational simplicity in the remainder of this section, we now abbreviate KoSSF simply
by $\eta$, that is, $\eta \equiv \eta(\dott;A,B)$.

In contrast to the usual (trace class) SSF theory, we have the following nonexistence result
for boundary values of modified perturbation determinants:

%%%%%%%%%%%%%%%%%%%%%%%%%%%%%%%%%%%
\begin{theorem}\label{t6.2}
There exists a pair of  self-adjoint operators $A, B$ $($in a complex, separable Hilbert
space$)$ such that $X=(A-B)\in\calI_2$, $\sigma(B)$ is an interval, and for
a.e.\ $\lambda\in\sigma(B)$, the nontangential limit $\lim_{z\to\lambda, \, z\in\bbC_+}
\det_2 (I+X(B-zI)^{-1})$ does not exist.
\end{theorem}
%%%%%%%%%%%%%%%%%%%%%%%%%%%%%%%%%%%

\begin{proof} By Theorems~\ref{t6.1} and \ref{newT7.1}, the proof reduces to the following 
statement: There exists a Riemann integrable $0\leq \eta\in L^1(\bbR;d\lambda)$ with 
support being an interval such that for a.e.\ $\lambda\in \supp\, (\eta)$, 
the nontangential limit
\begin{equation}
\lim_{\substack{z\to\lambda\\ z\in\bbC_+}}\int_{\bbR} \frac{\eta(\lambda)\,
d\lambda}{(\lambda-z)^2}
\label{6.5}
\end{equation}
does not exist.

First we note that the existence of the limit in \eqref{6.5} at the point $\lambda$
depends only on the behavior of $\eta(t)$ when $t$ varies in a small neighborhood
of $\lambda$. Thus, it suffices to construct $0\leq \eta\in L^1(\bbR;d\lambda)$ such that
the limits \eqref{6.5} do not exist for a.e.\ $\lambda\in(-1,1)$; by shifting and
scaling such a function $\eta$, one obtains the required statement for a.e.\
$\lambda\in\sigma(B)$.

Let us first obtain the required example of $\eta$ defined on the unit circle $\partial\bbD$,
and then transplant it onto the real line. By a well-known construction employing
either lacunary series or Rademacher functions (see \cite{Duren}, \cite[App.~A]{Du70},
\cite[I, p.~6]{Zy}), there exists a power series $f(z)=\sum_{n=1}^{\infty} c_n z^n$,
$\abs{z}\leq1$, such that $\sum_{n=1}^{\infty} \abs{c_n}<\infty$ and for a.e.\ $z\in\partial\bbD$,
the limit $\lim_{\zeta\to z}f'(\zeta)$ does not exist as $\zeta$ approaches $z$ from
inside of the unit disc along any nontangential trajectory. By construction, $\Ima (f)$
is continuous on $\partial\bbD$ and
$$
f(z)=\frac1\pi \int_{\partial\bbD} \frac{\Ima (f(\zeta))}{(\zeta-z)}\, d\zeta,
\quad
f'(z)=\frac1\pi \int_{\partial\bbD} \frac{\Ima (f(\zeta))}{(\zeta-z)^2}\, d\zeta,
\quad
\abs{z}<1.
$$
Let $a>-\min_{\zeta\in\partial\bbD} \Ima (f(\zeta))$ and set $v(\zeta)=\Ima (f(\zeta))+a$ if
$\abs{\arg \zeta}<\pi/2$ and $v(\zeta)=0$ otherwise. Then $v\geq 0$ and $v$ is
piecewise continuous (with the possible discontinuities only for $\arg(\zeta)=
\pm\pi/2$); in particular, $v$ is Riemann integrable. Again by a localization argument,
for a.e.\ $\theta\in(-\pi/2,\pi/2)$, the limit
$$
\lim_{z\to e^{i\theta}} \int_{\partial\bbD} \frac{\Ima (f(\zeta))}{(\zeta-z)^2}\, d\zeta
$$
does not exist as $z$ approaches $e^{i\theta}$ from inside of the unit disc along
any nontangential trajectory.

It remains to transplant $v$ from the unit circle onto the real line. Let
$t=i\frac{1-\zeta}{1+\zeta}$, $w=i\frac{1-z}{1+z}$, and $\eta(t)=v(\zeta(t))$.
Then $0\leq \eta\in L^1(\bbR;d\lambda)$, $\supp \, (\eta)\subset(-1,1)$, $\eta$ is Riemann
integrable, and
$$
 \int_{\partial\bbD} \frac{\Ima (f(\zeta))}{(\zeta-z)^2}\, d\zeta
 =
 -\frac{(w+i)^2}{2i}
 \int_{-1}^{1}\frac{\eta(\lambda) \, d\lambda}{(\lambda-w)^2}.
 $$
Thus, the limit \eqref{6.5} does not exist for a.e.\ $\lambda\in(-1,1)$.
 \end{proof}
%%%%%%%%%%%%%%%%%%%%%%%%%%%%%%%%%%%%

%%%%%%%%%%%%%%%%%%%%%%%%%%%%%%%%%%%%%%%%%%%%%%%%%%%%
\section{KoSSF for Unbounded Operators }\lb{s9}
%%%%%%%%%%%%%%%%%%%%%%%%%%%%%%%%%%%%%%%%%%%%%%%%%%%%

In this section we briefly discuss the question of existence of
KoSSF under the assumption
\begin{equation}\label{9.1}
[(A-z)^{-1} - (B-z)^{-1}] \in\calI_2
\end{equation}
instead of $(A-B)\in\calI_2$. This question was studied in \cite{Ne88} 
and \cite{Pe05} (see also \cite{Kop} for related issues).

First recall the invariance principle for the KrSSF. Assume that $A,B$ are bounded
self-adjoint operators and $(A-B)\in\calI_1$. Let $\varphi=\overline{\varphi}\in 
C^\infty(\bbR)$, $\varphi'\neq 0$ on $\bbR$. Then we have
\begin{equation}
\xi(\lambda;A,B)
=
\text{sign}\, (\varphi') \, \xi(\varphi(\lambda); \varphi(A),\varphi(B))
+ \text{const} \, \text{ for a.e.\ $\lambda\in\bbR$}.
\label{9.2}
\end{equation}
This is a consequence of Krein's trace formula \eqref{1.5}.
With an appropriate choice of normalization of KrSSF, the constant in the right-hand
side of \eqref{9.2} vanishes.

When both $(A-B)\in\calI_1$ and $[\varphi(A)-\varphi(B)]\in\calI_1$, formula \eqref{9.2} is
an easily verifiable \emph{identity}. But when $[\varphi(A)-\varphi(B)]\in\calI_1$ yet
$(A-B)\not\in\calI_1$, this formula can be regarded as a \emph{definition} of
$\xi(\dott;A,B)$.

In contrast to this, no explicit formula relating $\eta(\varphi(\cdot);\varphi(A), 
\varphi(B))$ to $\eta(\dott; A,B)$ is known. The reason is simple: The definition of 
$\eta$ involves not only a trace formula but a choice of interpolation $A(\theta)$ 
between $B$ and $A$. For bounded self-adjoint operators, the choice $A(\theta)= 
(1-\theta)B+\theta A$, $\theta\in [0,1]$, is natural. But when one only has \eqref{9.1}, 
what choice does one make? It is natural to define $A(\theta)$ by
\begin{equation}\lb{9.2a}
(A(\theta)-z)^{-1} =(1-\theta)(B-z)^{-1} + \theta(A-z)^{-1}, \quad \theta\in [0,1].
\end{equation}
For this to be self-adjoint, we need $z\in\bbR$, which means we should have some
real point in the intersection of the resolvent sets for $A$ and $B$. Even if there were
such a $z$, it is not unique and the interpolation will not be unique. Moreover, the
convexity that led to $\eta\geq 0$ may be lost. The net result is that the situation,
both after the work of others and our work, is less than totally satisfactory.

Let us discuss a certain surrogate of \eqref{9.2} for the KoSSF.
The formulas below are a slight variation on  the theme of the construction of \cite{Ne88}.

First assume that $A$ and $B$ are bounded operators and $X=(A-B) \in\calI_2$.
Let $\delta\subset\bbR$ be an interval which contains the spectra of $A$ and $B$ and
$\varphi\in C^\infty(\delta)$, $\varphi'\not=0$. Denote $a=\varphi(A)$, $b=\varphi(B)$,
$x=a-b$. By the Birman--Solomyak bound \eqref{1.6}, we have $x\in\calI_2$ and so
both $\eta(\dott; A,B)$ and $\eta(\dott; a,b)$ are well defined. Let us display the
corresponding trace formulas:
\begin{align}
\tr\bigg(f(A)-f(B)-\frac{d}{d\alpha}f(B+\alpha X)\bigg|_{\alpha=0}\bigg)
&=
\int_{\bbR} \eta(\lambda; A, B)f''(\lambda)\, d\lambda,
\label{9.3}
\\
\tr\bigg(g(a)-g(b)-\frac{d}{d\alpha}g(b+\alpha x)\bigg|_{\alpha=0}\bigg)
&=
\int_{\bbR} \eta(\mu;a,b)g''(\mu)\, d\mu.
\label{9.4}
\end{align}
Now suppose $f=g\circ\varphi$. In contrast to the corresponding calculation
for the KrSSF, the left-hand sides of \eqref{9.3} and \eqref{9.4} are, in general,
distinct. However, we can make the right-hand sides look similar if we introduce
the following modified KoSSF:
\begin{equation}
\widetilde \eta(\lambda; A,B) =
\eta(\varphi(\lambda); a,b)\, \frac{1}{\varphi'(\lambda)}
- \int_{\lambda_0}^\lambda \eta(\varphi(t); a,b)\left(\frac{1}{\varphi'(t)}\right)'dt.
\label{9.5}
\end{equation}
The choice of $\lambda_0$ above is arbitrary; it affects only the constant term in the
definition of $\widetilde\eta$.

By a simple calculation involving integration by parts, we get
\begin{equation}
\int_{\bbR} \eta(\mu; a,b)g''(\mu)\, d\mu
=
\int_{\bbR} \widetilde \eta(\lambda; A,B)f''(\lambda)\, d\lambda,
\quad f=g\circ \varphi\in C_0^\infty(\bbR).
\label{9.6}
\end{equation}
Combining \eqref{9.4} and \eqref{9.6}, we get the modified trace formula
\begin{equation}
\tr\bigg(f(A)-f(B)-\frac{d}{d\alpha}f\circ \varphi^{-1}(b+\alpha x)\bigg|_{\alpha=0}\bigg)
=
\int_{\bbR} \widetilde\eta(\lambda; A, B)f''(\lambda)\, d\lambda
\label{9.7}
\end{equation}
for all $f\in C_0^\infty(\bbR)$. Precisely as for the KrSSF, one can treat
\eqref{9.5} and \eqref{9.7} as the definition of a modified KoSSF $\widetilde \eta(\dott;A,B)$.

We consider an example of this construction which might be useful in applications.
Suppose that $A$ and $B$ are lower semibounded self-adjoint operators such that
for some (and thus for all) $z\in\bbC\backslash(\sigma(A)\cup\sigma(B))$
the inclusion \eqref{9.1} holds. Choose $E\in\bbR$ such that $\inf\sigma(A+E)>0$
and $\inf\sigma(B+E)>0$. Take $\varphi(\lambda)=\frac{1}{\lambda+E}$ and let
$a=(A+E)^{-1}$, $b=(B+E)^{-1}$, $x=a-b$. For $\lambda>-E$, define
\begin{align}
\begin{split}
\widetilde\eta(\lambda; A,B)
& = -\eta((\lambda+E)^{-1}; a,b)(\lambda+E)^2  \\
& \quad + 2\int_{-E}^\lambda \eta((t+E)^{-1};a,b)(t+E)\, dt.   \lb{9.7b}
\end{split}
\end{align}
Note that $\eta(\dott;a,b)$ is integrable and $\eta(\lambda;a,b)$ vanishes for large
$\lambda$ and therefore
the integral in \eqref{9.7b} converges. Moreover, this definition ensures that
$\widetilde\eta(\lambda;A,B)=0$ for $\lambda<\inf(\sigma(A)\cup\sigma(B))$.
Thus, it is natural to define
\begin{equation}\lb{9.7a}
\widetilde\eta(\lambda; A,B)=0 \, \text{ for } \, \lambda\leq -E.
\end{equation}
The above calculations prove the following result:
%%%%%%%%%%%%%%%%%
\begin{theorem}\label{t9.1}
Let $A$, $B$,  $a$, $b$, $x$ be as above.
Then there exists a function $\widetilde \eta(\dott; A,B)$ such that
\begin{equation}
\int_{\bbR} \widetilde\eta(\lambda;A,B)(\lambda+E)^{-4}d\lambda<\infty
\label{9.8}
\end{equation}
and $\widetilde \eta(\lambda;A,B)=0$ for $\lambda<\inf(\sigma(A)\cup\sigma(B))$
and for all $f\in C_0^\infty(\bbR)$ the following trace formula holds:
\begin{equation}
\tr\bigg(f(A)-f(B)-\frac{d}{d\alpha}f((b+\alpha x)^{-1}-E)\bigg|_{\alpha=0}\bigg)
=
\int_{\bbR} \widetilde \eta(\lambda; A,B)f''(\lambda)\, d\lambda.
\label{9.9}
\end{equation}
\end{theorem}
%%%%%%%%%%%%%%%%%%

We note that condition \eqref{9.8} does not fix the linear term in the definition
of $\widetilde\eta$ but \eqref{9.7a} does.

In \cite{Ne88}, a pair of self-adjoint operators $A, B$ was considered under the assumption
\eqref{9.1} alone (without the lower semiboundedness assumption).
Another regularization of $\eta(\dott;A,B)$  was suggested in this case.
The construction of \cite{Ne88} is more intricate than the above calculation and
uses KoSSF for unitary operators.

In \cite{Kop}, the assumption
$$
(A-B)|A-iI|^{-1/2}\in\calI_2
$$
was used. This assumption is intermediate between $(A-B)\in\calI_2$ and \eqref{9.1}.
Under this assumption, the trace formula \eqref{5.7} was proven with
$0\leq \eta\in L^1(\bbR;(1+\lambda)^{-\gamma}d\lambda)$ for any $\gamma> \f12$.

Finally, in \cite{Ne88}, the assumption
$$
(A-B)(A-iI)^{-1}\in\calI_2
$$
was used and formula \eqref{5.7} was proven with $\eta\in L^1(\bbR;(1+\lambda^2)^{-2}d\lambda)$.

Note that the difference between the last two results and Theorem~\ref{t9.1} is that
in Theorem~\ref{t9.1}, a modified trace formula \eqref{9.9} is proven rather than the
original formula \eqref{5.7}. Theorem~\ref{t9.1} is nothing but a change of variables
in the trace formula for resolvents, whereas the abovementioned results of \cite{Kop}
and \cite{Ne88} require some work.

%%%%%%%%%%%%%%%%%%%%%%%%%%%%%%%%%%%%
\section{The Case of Unitary Operators} \lb{s10}
%%%%%%%%%%%%%%%%%%%%%%%%%%%%%%%%%%%%

In this section, we want to briefly discuss a definition of $\eta$ for a pair of unitaries.
Once again, there is an issue of interpolation. If $A$ and $B$ are the unitaries,
\begin{equation}\lb{10.1}
A(\theta)=(1-\theta)B+\theta A, \quad \theta \in [0,1],
\end{equation}
is not unitary, so we cannot define $f(A(\theta))$ for arbitrary $C^\infty$-functions on
$\partial\bbD=\{z\in\bbC \, | \, \abs{z}=1\}$. Neidhardt \cite{Ne88} (see also \cite{Pe05})
discussed one way of interpolating by writing $A=e^C$, $B=e^D$ for suitable $C$ and $D$ and
interpolating, but there is considerable ambiguity in how to choose $C,D$ as well as
whether to look at $e^{\theta C+ (1-\theta)D}$ or $e^{(1-\theta)D} e^{\theta C}$, etc.

Here, with Szeg\H{o}'s theorem as background \cite{GPS}, we want to discuss an alternative
to Neidhardt's approach.

%%%%%%%%%%%%%
\begin{lemma}\lb{L10.1} Let $A,B$ be unitary with $(A-B)\in\calI_2$. Then for any
$n=0,1,2,\dots$,
\begin{equation}   \lb{10.2}
\bigg(A^n -B^n - \f{d}{d\theta}\, A(\theta)^n\bigg|_{\theta=0}\bigg)\in \calI_1.
\end{equation}
We have
\begin{equation}\lb{10.3}
\biggl\| \left. A^n-B^n -\f{d}{d\theta}\, A(\theta)^n\right|_{\theta=0}\biggr\|_{\calI_1} \leq
\f{n(n-1)}{2}\, \|A-B\|_{\calI_2}^2.
\end{equation}
In fact,
\begin{equation}\lb{10.4}
\text{LHS of \eqref{10.3}} = o(n^2).
\end{equation}
\end{lemma}
%%%%%%%%%%%%%
\begin{proof} Let $X=A-B$. Then, by telescoping,
\begin{equation}\lb{10.5}
A(\theta)^n -B^n =\sum_{j=0}^{n-1} A(\theta)^j (\theta X) B^{n-1-j}.
\end{equation}
Thus, since $\|A(\theta)\|\leq 1$, $\|B\|=1$,
\begin{equation}\lb{10.6}
\|A(\theta)^n -B^n\|_{\calI_2} \leq n\abs{\theta}\, \|X\|_{\calI_2}
\end{equation}
and, of course,
\begin{equation}\lb{10.7}
\|A(\theta)^n-B^n\|\leq 2.
\end{equation}

Dividing \eqref{10.5} by $\theta$ and taking $\theta$ to zero yields
\[
\f{d}{d\theta}\, A(\theta)^n =\sum_{j=0}^{n-1} B^j\, X\, B^{n-1-j},
\]
so
\begin{equation}\lb{10.8}
\text{LHS of \eqref{10.2}} =\sum_{j=0}^{n-1}\, (A^j-B^j)X\, B^{n-1-j}.
\end{equation}

\eqref{10.3} is immediate since \eqref{10.8} and \eqref{10.6} implies
\begin{equation}\lb{10.9}
\text{LHS of \eqref{10.3}} \leq \|X\|_2^2 \biggl(\, \sum_{j=0}^{n-1} j\biggr).
\end{equation}

To get \eqref{10.4}, we write $X=X_\veps^{(1)} + X_\veps^{(2)}$ where
$\|X_\veps^{(2)}\|_{\calI_2}
\leq\veps$ and $\|X_\veps^{(1)}\|_{\calI_1} <\infty$. Thus, \eqref{10.8} implies
\[
\text{LHS of \eqref{10.3}} \leq \veps\|X\|_{\calI_1}\, \f{n(n-1)}{2} + 2n\|X_\veps^{(1)}\|_{\calI_1}
\]
using \eqref{10.7} instead of \eqref{10.6}. Dividing by $n^2$, taking $n\to\infty$, and
then $\veps\downarrow 0$, show
\[
\limsup_{n\to\infty} n^{-2} \text{ LHS of \eqref{10.3}} =0.
\qedhere
\]
\end{proof}
%%%%%%%%%%%%%

%%%%%%%%%%%%%
\begin{theorem}\lb{T10.2} Let $A$ and $B$ be unitary so $(A-B)\in\calI_2$. Then there
exists a real distribution $\eta(\lambda;A,B)$ on $\partial\bbD$ so that for any
polynomial $P(z)$, $\big[P(A)-P(B)-\f{d}{d\theta} P(A(\theta))\big]\big|_{\theta=0}\in
\calI_2$ and
\begin{equation}\lb{10.10}
 \tr\biggl( \left.P(A)-P(B) -\f{d}{d\theta}\, P(A(\theta))\right|_{\theta=0}\biggr) =
\int_0^{2\pi} P''(e^{i\theta}) \eta (e^{i\theta};A,B)\, \f{d\theta}{2\pi}.
\end{equation}
Moreover, the moments of $\eta$ satisfy
\begin{equation}\lb{10.11}
\int_0^{2\pi} e^{in\theta} \eta(e^{i\theta})\, \f{d\theta}{2\pi} \underset{\abs{n}\to\infty}{=} o(1).
\end{equation}
\end{theorem}
%%%%%%%%%%%%%

%%%%%%%%%%%%%
\begin{remarks} 1. As usual, we use
$\int_0^{2\pi} f(e^{i\theta}) \eta(e^{i\theta}) \f{d\theta}{2\pi}$
as shorthand for the distribution $\eta$ acting on the function $f$.

\smallskip
2. As we will discuss, $\eta$ is determined by \eqref{10.10} up to three real constants
in an affine term.

\smallskip
3. For a sharp condition on the class of functions for which Neidhardt's version
of Koplienko's trace formula for unitary operators holds, we refer to Peller \cite{Pe05}.
\end{remarks}
%%%%%%%%%%%%%
\begin{proof}
Let $c_n$, $n\in\bbZ$, be defined by
\begin{equation}\lb{10.12}
c_n =  \begin{cases} 0, & n=0,1, \\
[n(n-1)]^{-1} \,  \tr \bigl(A^n-B^n - \f{d}{d\theta}\, A(\theta)^n
\big|_{\theta=0}\big), & n\geq 2, \\
\overline{c_{-n}}, & n\leq -1.
\end{cases}
\end{equation}
By Lemma \ref{L10.1}, $c_n=o(1)$ as $n\to\infty$, so there is a distribution
$\eta=\eta(\dott;A,B)$ satisfying
\begin{equation} \lb{10.13}
c_n=\int_0^{2\pi} e^{i (n-2) \theta} \eta(e^{i\theta})\, \f{d\theta}{2\pi}, \quad n\geq 2.
\end{equation}

By \eqref{10.12}, we have \eqref{10.10} for $P(z)=z^n$ for $n\geq 2$ and both sides are
zero for $P(z)=z^m$, $m=0,1$. Thus, \eqref{10.10} holds for all polynomials.
\end{proof}
%%%%%%%%%%%%%%

For any $c_0\in\bbR$, $c_1\in\bbC$, we can add $c_0 + c_1 e^{i\theta} + \bar c_1 e^{-i\theta}$
to $\eta$ without changing the right-hand side of \eqref{10.10}. We wonder if $\eta$ is always
in $L^1(\partial\bbD)$ with $\eta\geq 0$ for some choice of $c_0$ and $c_1$. The condition
$c_n\to 0$ is, of course, consistent with $\eta\in L^1(\partial\bbD)$.

%%%%%%%%%%%%%%%%%%%%%%%%%%%%%%%%%%%
\section{Open Problems and Conjectures} \lb{s11}
%%%%%%%%%%%%%%%%%%%%%%%%%%%%%%%%%%%

While we have found some new aspects of $\eta$ here and summarized much of the prior
literature, there are many open issues. The most important one concerns properties of
$\eta$ and the invariance of the a.c.\ spectrum:

\begin{conjecture}\lb{Con11.1} Suppose $A,B$ are self-adjoint with $(A-B)\in\calI_2$ and
that on some interval $(a,b)\subset\sigma(A)\cap\sigma(B)$, we have $\eta(\dott;A,B)$
and $\eta(\dott;B,A)$ are of bounded variation with distributional derivatives in
$L^p ((a,b);d\lambda)$ on $(a,b)$ for some $p>1$. Then
$\sigma_\ac(A)\cap(a,b)=\sigma_\ac (B) \cap (a,b)$.
\end{conjecture}

In the appendix, we prove the invariance for $\calI_1$-perturbations using boundary
values of $\det((A-z)(B-z)^{-1})$. When $\eta$ has the properties in the conjecture,
$\det_2((A-z)(B-z)^{-1})$ has boundary values and we hope those can be used to get
the invariance of a.c.\ spectrum. While we made the conjecture assuming control of
$\eta(\dott;A,B)$ and $\eta(\dott;B,A)$, we wonder if only one suffices. Similarly,
we wonder if $L^p$, $p>1$, can be replaced by the weaker condition that the
derivative is a sum of an $L^1$-piece and the Hilbert transform of an $L^1$-piece.

\begin{OQ2} Is the $\eta$ we constructed in Section~\ref{s10} for the unitary case
an $L^1(\partial\bbD)$ function?
\end{OQ2}

\begin{OQ3} Is the class $\eta(\calI_2)$ introduced in Section~\ref{s5a} all of
$L^1(\bbR;d\lambda)$ $($of compact support\,$)$, or only the Riemann integrable
functions, or something in between?
\end{OQ3}

\begin{OQ4} Is the class $\eta(\calI_1)$ all functions of bounded variation or a
subset, and if so, what subset?
\end{OQ4}

%%%%%%%%%%%%%%%%%%%%%%%%%%%%%%%%%%%%
%\appendix
\section*{Appendix: On the KrSSF $\xi(\dott;A,B)$} \lb{App}
\renewcommand{\theequation}{A.\arabic{equation}}
\renewcommand{\thetheorem}{A.\arabic{theorem}}
\setcounter{theorem}{0}
\setcounter{equation}{0}
%%%%%%%%%%%%%%%%%%%%%%%%%%%%%%%%%%%%

Both for comparison and because the Krein spectral shift (KrSSF) is needed in our
construction of the KoSSF, we present the basics of the KrSSF here. Most of the
results in this appendix are known (see, e.g., \cite[Sect.\ 19.1.4]{BW83}), \cite{BS75},
\cite{BY93}, \cite{Kr53}, \cite{Kr62}, \cite{Kr83}, \cite{SM94}, \cite{Vo87},  
\cite[Ch.\ 8]{Ya92}, \cite{Ya07} and the references therein) so this appendix is 
largely pedagogical, but our argument proving the invariance of a.c.\ spectrum 
under trace class perturbations at the end of this appendix is new. Moreover, 
we fill in the details of an approach sketched in \cite[Ch.~11]{STI} exploiting 
the method Gesztesy--Simon \cite{GS95} used to construct the rank-one KrSSF. 
Most approaches define $\xi$ via perturbation determinants.

We will need the following strengthening of Theorem~\ref{T2.2}:

%%%%%%%%%%%
\begin{theorem}\lb{TA.1} Let $f$ be a function of compact support whose Fourier
transform $\widehat f$ satisfies \eqref{2.5} for $n=1$ {\rm{(}}in particular,
$f$ can be $C^{2+\veps}(\bbR)${\rm{)}}. Then,
\begin{SL}
\item[{\rm{(a)}}] For any bounded self-adjoint operators $A,B$ with $(A-B)\in\calI_1$,
$(f(A)-f(B)) \in\calI_1$. Moreover,
\begin{equation}\lb{A.1}
\|f(A)-f(B)\|_{\calI_1} \leq \| k\widehat f\|_1 \|A-B\|_{\calI_1},
\end{equation}
where
\begin{equation}\lb{A.2}
\|k\widehat f\|_1 \leq \int_{\bbR} k\abs{\widehat f(k)}\, dk.
\end{equation}
\item[{\rm{(b)}}] Let $B_n,B$, $n\in\bbN$, be uniformly bounded self-adjoint operators
such that $B_n\underset{n\to\infty}{\longrightarrow}B$ strongly. Let $X_n,X$, $n\in\bbN$,
be a sequence of self-adjoint trace class operators such that $\|X-X_n\|_{\calI_1}
\underset{n\to\infty}{\longrightarrow} 0$. Then,
\begin{equation}\lb{A.3}
\tr (f(B_n+X_n)-f(B_n)) \underset{n\to\infty}{\longrightarrow}  \tr (f(B+X)-f(B)).
\end{equation}
\end{SL}
\end{theorem}
%%%%%%%%%%%%%
\begin{proof} (a) is immediate from Proposition~\ref{P2.1} which implies
\begin{align}  \lb{A.4}
\begin{split}
& f(A)-f(B)  \\
& \quad =(2\pi)^{-1/2} \int_{\bbR} ik\widehat f(k) \biggl[\int_0^1 e^{i\beta kA}(A-B)
e^{i(1-\beta) kB}\, d\beta\biggr]\, dk.
\end{split}
\end{align}
This also implies (b) via the dominated convergence theorem, continuity of the functional
calculus (so $C_n \underset{n\to\infty}{\longrightarrow}  C$ strongly implies
$e^{it C_n} \underset{n\to\infty}{\longrightarrow}  e^{itC}$ strongly),
and the fact that if $X_n \underset{n\to\infty}{\longrightarrow}  X$ in $\calI_1$ and
$C_n \underset{n\to\infty}{\longrightarrow}  C$ strongly (with $C_n, C$ uniformly bounded),
then $\tr(C_nX_n) \underset{n\to\infty}{\longrightarrow} \tr (CX)$. This latter fact comes from
\begin{align*}
\abs{\tr(C_nX_n-CX)} &\leq \abs{\tr (C_n(X_n-X)-(C-C_n)X)} \\
&\leq \|C_n\|\, \|X_n-X_1\|_{\calI_1} + \abs{\tr((C-C_n)X)}
\end{align*}
and if $X=\sum_{m\in\bbN} \mu_m(X)\langle \varphi_m, \dott\rangle\psi_m$, then
$$
\abs{\tr((C_n-C)X)} \leq \sum_{m\in\bbN} \mu_m(X) \abs{\langle\varphi_m, (C_n-C)\psi_m\rangle}
\underset{n\to\infty}{\longrightarrow}  0
$$
by the dominated convergence theorem.
\end{proof}
%%%%%%%%%%%%%

Part (a) in Theorem~\ref{TA.1}, in a slightly more general form, is stated and 
proved in \cite[p.~141]{Kr83}.

Now let $B$ be a bounded self-adjoint operator and $\varphi$ a unit vector. 
For $\alpha\in \bbR$, define
\begin{equation}\lb{A.5}
A_\alpha =B+\alpha (\varphi,\dott)\varphi
\end{equation}
and for $z\in\bbC\backslash \bbR$,
\begin{align}
F_\alpha(z) &= (\varphi, (A_\alpha -z)^{-1}\varphi),   \lb{A.6} \\
G_\alpha(z) &= 1+\alpha F_0(z).   \lb{A.7}
\end{align}
The resolvent formula implies (see \cite[Sect.\ 11.2]{STI})
\begin{equation}\lb{A.8}
F_\alpha(z) =\f{F_0(z)}{1+\alpha F_0(z)}, \quad z\in\bbC\backslash \bbR,
\end{equation}
and that
\begin{equation} \lb{A.9} 
\begin{split}
 (A_\alpha &-z)^{-1} -(B-z)^{-1} \\ 
\quad & = -\f{\alpha}{1+\alpha F_0(z)}\,
((B-\bar z)^{-1}\varphi, \dott)(B-z)^{-1}\varphi,  
\quad  z\in\bbC\backslash \bbR,    
\end{split} 
\end{equation}
implying
\begin{align}
\begin{split}
\tr ((B-z)^{-1} -(A_\alpha -z)^{-1}) &= \f{\alpha}{1+\alpha F_0(z)}\,
(\varphi, (B-z)^{-2}\varphi)     \\
&= \f{d}{dz}\, \log (G_\alpha(z)), \quad z\in\bbC\backslash \bbR.   \lb{A.10}
\end{split}
\end{align}

%%%%%%%%%%%%%
\begin{theorem}\lb{TA.2} Let $B$ be a bounded self-adjoint operator and $A_\alpha$ 
given by \eqref{A.5} for $\alpha\in\bbR$ and $\varphi$ with $\|\varphi\|=1$. Then 
for a.e.\ $\lambda\in\bbR$, 
\begin{equation}\lb{A.11}
\xi_\alpha(\lambda) = \f{1}{\pi}\, \lim_{\veps\downarrow 0}\, \arg (G_\alpha (\lambda +
i\veps))
\end{equation}
exists and satisfies
\begin{SL}
\item[{\rm{(i)}}]
\begin{equation}\lb{A.12}
0 \leq \pm\xi_\alpha (\dott) \leq 1 \, \text{ if } \, 0 < \pm\alpha.
\end{equation}

\item[{\rm{(ii)}}] $\xi_\alpha(\lambda)=0$ if $\lambda\leq \min(\sigma(A_\alpha)\cup\sigma(B))$
or $\lambda\geq\max(\sigma(A_\alpha)\cup\sigma(B))$.

\item[{\rm{(iii)}}]
\begin{equation}\lb{A.13}
\int\abs{\xi_\alpha(\lambda)}\, d\lambda =\abs{\alpha}
\end{equation}

\item[{\rm{(iv)}}] For any $z\in\bbC\backslash \bbR$,
\begin{equation}\lb{A.14}
G_\alpha(z) =\exp\bigg( \int_{\bbR} (\lambda -z)^{-1} \xi_\alpha (\lambda) \, d\lambda
\bigg).
\end{equation}

\item[{\rm{(v)}}]
\begin{equation}\lb{A.15}
\det ((A_\alpha -z)(B-z)^{-1}) =G_\alpha (z).
\end{equation}

\item[{\rm{(vi)}}] For any $z\in\bbC\backslash \bbR$,
\begin{equation}\lb{A.16}
\tr((B-z)^{-1} -(A_\alpha -z)^{-1}) =\int_{\bbR} (\lambda -z)^{-2}
\xi_\alpha (\lambda)\, d\lambda.
\end{equation}

\item[{\rm{(vii)}}] For any $f$ satisfying the hypotheses of Theorem~\ref{TA.1},
\begin{equation}\lb{A.17}
\tr(f(A_\alpha)-f(B)) =\int_{\bbR} f'(\lambda) \xi_\alpha (\lambda)\, d\lambda.
\end{equation}
\end{SL}
\end{theorem}
%%%%%%%%%%%%%

%%%%%%%%%%%%%
\begin{remarks} 1. This theorem and its proof are essentially the same as the
starting point of Krein's construction in
\cite{Kr53} (see also \cite[p.~134--136]{Kr83} or \cite[Sect.\ 3]{BY93}).

\smallskip
2. In \eqref{A.11}, $\arg(G_\alpha(z))$ is defined uniquely for
$\Ima (z)>0$ by demanding continuity in $z$ and
\begin{equation}\lb{A.18x}
\lim_{y\uparrow\infty}\, \arg(G_\alpha(iy)) =0.
\end{equation}
For $\Ima (z)<0$ one has $\overline{G_\alpha(\overline{z})}=G_\alpha(z)$.

\smallskip
3. By \eqref{A.9}, $(A_\alpha-z)(B-z)^{-1}$ is of the form $I+$ rank one, and so lies in
$I+\calI_1$. The $\det(\cdot)$ in \eqref{A.15} is the Fredholm determinant (see 
\cite[Ch.\ 3]{STI}). This is the same as the finite-dimensional determinant $\det(C)$ 
for $I+D$ with $D$ finite rank and $C=(I+D)\restriction \calK$ where $\calK$ is any 
finite-dimensional space containing $\ran (D)$ and $(\ker (D))^\perp$.

\smallskip
4. The exponential Herglotz representation basic to this proof goes back to
Aronszajn and Donoghue \cite{AD56}.

\smallskip
5. Comparing \eqref{A.17} and \eqref{1.5}, one concludes
$$
\xi_\alpha(\dott)=\xi(\dott;A,B).
$$
\end{remarks}
%%%%%%%%%%%%%

%%%%%%%%%%%%%
\begin{proof} By the spectral theorem, there is a probability measure $d\mu_\alpha
(\lambda)$ such that
\begin{equation}\lb{A.18}
F_\alpha(z) =\int_{\bbR} \f{d\mu_\alpha (\lambda)}{\lambda -z}.
\end{equation}
In particular,
\begin{equation}\lb{A.19}
\Ima (F_0(z)) >0 \, \text{ if } \, \Ima (z)>0,
\end{equation}
so on $\bbC_+=\{z\in\bbC\,|\, \Ima(z)>0\}$,
\begin{equation}\lb{A.20}
\pm\Ima (G_\alpha(z))>0 \, \text{ if } \, \pm\alpha >0.
\end{equation}
Since $G_\alpha (iy) \to 1$, as $y \uparrow \infty$, we can define
\[
\log(G_\alpha(z)) =H_\alpha (z)
\]
on $\bbC_+$ uniquely if we require
$H_\alpha (iy)\underset{y\uparrow \infty}{\longrightarrow} 0$. By \eqref{A.20},
\begin{equation}\lb{A.21}
0 < \pm\Ima (H_\alpha (\dott)) \leq \pi.
\end{equation}

By the general theory of Herglotz functions (see, e.g., \cite{Ar57}, \cite{AD56}),
the limit in \eqref{A.11}
exists and \eqref{A.12} holds by \eqref{A.21}. \eqref{A.21} also implies that the
limiting measure $\wlim_{\veps\downarrow 0} \pm\f{1}{\pi}
\Ima (H_\alpha (\lambda+i\veps))\, d\lambda$ in the Herglotz representation theorem
is purely absolutely continuous, hence \eqref{A.14} holds.

\eqref{A.16} then follows from \eqref{A.14} and \eqref{A.10}.

Since
\begin{equation}\lb{A.22}
F_\alpha(z) \underset{z\to\infty}{=} -z^{-1} + O(z^{-2}),
\end{equation}
\eqref{A.17} implies
\begin{equation}\lb{A.23}
G_\alpha(z) \underset{z\to\infty}{=} 1-\alpha z^{-1} + O(z^{-2}),
\end{equation}
and thus \eqref{A.14} implies
\begin{equation}\lb{A.24}
\int_{\bbR} \xi_\alpha (\lambda)\, d\lambda =\alpha,
\end{equation}
which, given \eqref{A.12}, implies \eqref{A.13}. This proves everything except
the parts (ii), (v), and (vii).

To prove \eqref{A.15}, we note that with $P_\varphi =(\varphi,\dott)\varphi$, we have
\[
(A_\alpha -z)(B-z)^{-1}=I+\alpha P_\varphi (B-z)^{-1},
\]
which, since $P_\varphi$ is rank one, implies
\begin{align*}
\det((A_\alpha-z)(B-z)^{-1}) &= 1+\tr (\alpha P_\varphi (B-z)) \\
&= 1+\alpha F_0(z) \\
&= G_\alpha (z).
\end{align*}

Let us prove (ii) for $\alpha >0$. The proof of $\alpha <0$ is similar. Let $a=\min
(\sigma(B))$, $b=\max(\sigma(B))$. Then, by \eqref{A.18},
\[
F'_0(x) =\int_{\bbR} \f{d\mu_0 (\lambda)}{(\lambda-x)^2} >0
\]
on $(-\infty, a)\cup(b,\infty)$ and $F_0 \underset{x\to\pm\infty}{\longrightarrow} 0$.
Thus, $F>0$ on $(-\infty,a)$ and $F<0$ on $(b,\infty)$. Let $f=\lim_{x\downarrow b} F(x)$
which may be $-\infty$. If $1+\alpha f<0$, there is a unique $c$ with $1+\alpha F_0(c)=0$,
and then $G_\alpha$ is positive on $(c,\infty)$. By \eqref{A.8}, $F_\alpha(z)$ is analytic
away from $(a,b) \cup\{c\}$. Thus, $\sigma(A_\alpha)\in (a,b)\cup\{c\}$ and $c=\max(\sigma(A_\alpha),
\sigma(B))$, so (ii) says that $\xi_\alpha(\lambda)=0$ on $(-\infty,b)$ and $(c,\infty)$.
Since $G_\alpha(x) >0$ there and $0<\arg( G_\alpha(z+i\veps))<\pi$, we see that
$\xi_\alpha (x) =0$ on these intervals.

Finally, we turn to (vii). Since $B^n-A_\alpha^n$ can be written as a telescoping series,
it is trace class and
\begin{equation}\lb{A.25}
\|B^n-A_\alpha^n\|_{\calI_1} \leq n\,[\sup(\|A_\alpha\|,\|B\|)]^{n-1}\|B-A\|_{\calI_1}.
\end{equation}
Thus, both sides of \eqref{A.16} are analytic about $z=\infty$, so identifying Taylor
coefficients,
\begin{equation}\lb{A.26}
\tr (B^n-A_\alpha^n ) =\int_{\bbR} n\lambda^{n-1} \xi_\alpha(\lambda)\, d\lambda.
\end{equation}
Summing Taylor series for $e^{z\lambda}$, using \eqref{A.25} and \eqref{A.26} proves
$(e^{zB}-e^{zA_\alpha})\in\calI_1$ and
\begin{equation}\lb{A.27}
\tr (e^{zB}-e^{zA_\alpha}) = z \int_{\bbR} e^{z\lambda} \xi_\alpha(\lambda)
\, d \lambda.
\end{equation}
This leads to \eqref{A.17} by using \eqref{A.4}.
\end{proof}
%%%%%%%%%%%%%

In extending this, the following uniqueness result will be useful:

%%%%%%%%%%%%%
\begin{proposition}\lb{PA.3} Suppose $A$ and $B$ are bounded self-adjoint operators
and $(A-B)\in\calI_1$. Suppose $\xi_j\in L^1(\bbR;d\lambda)$ for $j=1,2$, and for all
$f\in C_0^\infty (\bbR)$,
\begin{equation}\lb{A.28}
\tr(f(A)-f(B)) =\int_{\bbR} f' (\lambda)\xi_j(\lambda)\, d\lambda.
\end{equation}
Then $\xi_1=\xi_2$. Moreover, if $(a,b)\subset\bbR\backslash \sigma(A)\cup\sigma(B)$,
$\xi_j(\cdot)$ is an integer on $(a,b)$, and if $a=-\infty$ or $b=\infty$, it is zero 
on $(a,b)$, and so $\xi_j$ has compact support. 
\end{proposition}
%%%%%%%%%%%%
\begin{proof} By \eqref{A.28}, the distribution $\xi_1-\xi_2$ has vanishing
distributional derivative, so is constant. Since it lies in $L^1(\bbR;d\lambda)$, 
it must be zero.

If $f\in C_0^\infty((a,b))$, $f(A)=f(B)=0$, so $\xi'_j$ has zero derivative on
$(a,b)$ and so is constant. If $a=-\infty$ or $b=\infty$, the constant must be zero
since $\xi_j\in L^1(\bbR;d\lambda)$. Now pick $f$ which is supported on $(c,(a+b)/2)$
for some $c<d<\min(\sigma(A)\cup \, \sigma(B))$ with $f=1$ on $(d,(3 a + b)/4)$.
Thus, the right-hand side of \eqref{A.28} is the negative of the constant value of $\xi_j$
on $(a,b)$, while the left-hand side is the trace of a trace class difference of projections
which is always an integer (see \cite{AvSS, Eff}).
\end{proof}
%%%%%%%%%%%%

%%%%%%%%%%%%
\begin{theorem}\lb{TA.4} For any pair of bounded self-adjoint operators $A,B$ with
$(A-B)$ of finite rank, there exists a function, $\xi(\dott;A,B)$ such that the following hold:
\begin{SL}
\item[{\rm{(i)}}] \eqref{A.17} holds for any $f$ satisfying the hypotheses of
Theorem~\ref{TA.1}.

\item[{\rm{(ii)}}]
\begin{equation}\lb{A.29}
\abs{\xi(\dott;A,B)}\leq\rank(A-B).
\end{equation}

\item[{\rm{(iii)}}]
\begin{equation}\lb{A.30}
\int_{\bbR} \abs{\xi (\lambda;A,B)}\, d\lambda \leq \|A-B\|_{\calI_1}.
\end{equation}

\item[{\rm{(iv)}}] For $z\in\bbC \backslash \bbR$, one has
\begin{equation}\lb{A.31}
\det((A-z)(B-z)^{-1}) =\exp\bigg(\int_{\bbR} (\lambda-z)^{-1}
\xi(\lambda) \, d\lambda\bigg).
\end{equation}

\item[{\rm{(v)}}] $\xi(\lambda)=0$ for $\lambda\leq\min(\sigma(A)\cup\sigma(B))$ or
$\lambda\geq \max(\sigma(A)\cup\sigma(B))$.

\item[{\rm{(vi)}}] If $(A-B)$ and $(B-C)$ are both finite rank,
\begin{equation}\lb{A.32}
\xi (\dott;A,C)= \xi(\dott;A,B) + \xi (\dott;B,C).
\end{equation}
\end{SL}
\end{theorem}
%%%%%%%%%%%%%
\begin{proof} If $(A-B)$ has rank $n$, we can find $A_0=A,A_1, \dots, A_n=B$ so
$(A_{j+1}-A_j)$ has rank one, and
\begin{equation}\lb{A.33}
\sum_{j=0}^{n-1}\, \|A_{j+1}-A_j\|_{\calI_1} = \|B-A\|_{\calI_1}.
\end{equation}
We define
\begin{equation}\lb{A.34}
\xi(\dott;A,B)=\sum_{j=0}^{n-1} \xi(\dott;A_j, A_{j+1}),
\end{equation}
where $\xi (\dott;A_j, A_{j+1})$ is constructed via Theorem~\ref{TA.2}. \eqref{A.17}
holds by telescoping and the rank-one case. \eqref{A.29} and \eqref{A.30} follow from
\eqref{A.12}, \eqref{A.13}, and \eqref{A.33}.

\eqref{A.31} follows from
\[
(A-z)(B-z)^{-1} = [(A_0-z)(A_1-z)^{-1}] [(A_1-z)(A_2-z)^{-1}] \dots
\]
using
\[
\det((1+X_1)(1+X_2)) =\det (1+X_1)\det (1+X_2)
\]
for $X_1, X_2\in\calI_1$.

Item (v) is proven in Proposition~\ref{PA.3}. Item (vi) follows from the uniqueness in
Proposition~\ref{PA.3}.
\end{proof}
%%%%%%%%%%%%

Theorem \ref{TA.4} is essentially the same as Theorem~3 in \cite{Kr53} (see also
\cite{Kr83} and \cite{BY93}).

%%%%%%%%%%%%
\begin{corollary}\lb{CA.5} If $A,A'$ are both finite rank perturbations of $B$ with all
three operators self-adjoint, we have
\begin{equation}\lb{A.35}
\int_{\bbR} \abs{\xi(\lambda;A,B)-\xi(\lambda;A',B)}\, d\lambda \leq \|A-A'\|_{\calI_1}.
\end{equation}
\end{corollary}
%%%%%%%%%%%%
\begin{proof} By \eqref{A.32},
\[
\xi(\dott;A,B)- \xi (\dott;A',B)=\xi(\dott; A,A').
\]
Thus, \eqref{A.35} follows from \eqref{A.30}.
\end{proof}
%%%%%%%%%%%%

This yields the principal result on existence and properties of the KrSSF
(see \cite{Kr53} or \cite{Kr83}).

%%%%%%%%%%%%
\begin{theorem}\lb{TA.5} Let $A,B$ be bounded self-adjoint operators with $(A-B) \in\calI_1$.
Then,
\begin{SL}
\item[{\rm{(i)}}] There exists a unique function $\xi(\dott;A,B)\in L^1 (\bbR;d\lambda)$
such that \eqref{A.17} holds for any $f$ satisfying the hypotheses of Theorem~\ref{TA.1}.

\item[{\rm{(ii)}}]
\begin{equation}\lb{A.36}
\int_{\bbR} \abs{\xi (\lambda;A,B)}\, d\lambda \leq \|A-B\|_{\calI_1}.
\end{equation}

\item[{\rm{(iii)}}] \eqref{A.31} holds.

\item[{\rm{(iv)}}] $\xi(\lambda)=0$ if $\lambda\leq\min(\sigma(A)\cup\sigma(B))$ or
$\lambda \geq \max(\sigma(A)\cup\sigma(B))$.

\item[{\rm{(v)}}] If $(A-B)$ and $(B-C)$ are both trace class, \eqref{A.32} holds.

\item[{\rm{(vi)}}] If $(A-B)$ and $(A'-B)$ are trace class, \eqref{A.35} holds.
\end{SL}
\end{theorem}
%%%%%%%%%%%%
\begin{proof} Find $A_n$ so $(A_n -B) \underset{n\to\infty}{\longrightarrow} (A-B)$ in
$\calI_1$ and $(A_n-B)$ is finite rank. By \eqref{A.35}, $\xi(\dott;A_n,B)$ is Cauchy in
$L^1(\bbR)$ so converges to an $L^1(\bbR)$ function by \eqref{A.35}. Thus, items (i), (ii),
(iii), (v), and (vi) hold by taking limits (using $\|\cdot\|_{\calI_1}$-continuity of the
mapping $C\to\det(I+C$). Uniqueness and (iv) follow from Proposition~\ref{PA.3}.
\end{proof}
%%%%%%%%%%%%

We refer to \cite{Pe85} (see also \cite{Pe90}) for a description of a class of functions $f$
for which this theorem holds.

We note that there are interesting extensions of the trace formula \eqref{A.17} to classes of 
operators $A, B$ different from self-adjoint or unitary operators. While we cannot possibly 
list all such extensions here, we refer, for instance, to Adamjan and Neidhardt 
\cite{AN90}, Adamjan and Pavlov \cite{AP80}, Jonas \cite{Jo88}, \cite{Jo99}, 
Krein \cite{Kr89},  Langer \cite{La65}, Neidhardt \cite{Ne87}, \cite{Ne88a},
Rybkin \cite{Ry96}, Sakhnovich \cite{Sa68}, and the literature cited therein.

%%%%%%%%%%%%
\begin{theorem}\lb{TA.6} Let $B_n,B$, $n\in\bbN$, be uniformly bounded self-adjoint
operators such that $B_n\underset{n\to\infty}{\longrightarrow}B$ strongly.
Let $X_n,X$, $n\in\bbN$, be a sequence of self-adjoint trace class operators
such that $\|X-X_n\|_{\calI_1}\underset{n\to\infty}{\longrightarrow} 0$. Then for any
continuous function, $g$,
\begin{equation}\lb{A.37a}
\int_{\bbR} g(\lambda)\xi(\lambda;B_n+X_n,B_n)\, d\lambda
\underset{n\to\infty}{\longrightarrow}  \int_{\bbR} g(\lambda) \xi (\lambda; B+X,B)\, d\lambda.
\end{equation}
\end{theorem}
%%%%%%%%%%%%

\begin{proof} By Theorem~\ref{TA.1}, we have \eqref{A.37a} for $g\in C_0^\infty (\bbR)$. 
Note that the ${\|\cdot\|_\infty}$-norm closure of $C_0^\infty$ includes the continuous 
functions of compact support. Thus, by an approximation argument using uniform $L^1(\bbR; d\lambda)
$-bounds on $\xi$, we get \eqref{A.37a} for continuous functions of compact support.
Since $\xi (\lambda;A,B)=0$ for $\lambda\in [-\max (\|A\|,\|B\|), \max(\|A\|,\|B\|)]$, 
the result for continuous functions of compact support extends to any continuous function.
\end{proof}
%%%%%%%%%%%%%

We want to note the following. Define
\[
\xi(\calI_1) =\{\xi (\dott; A,B) \, | \, A, B \text{ bounded and self-adjoint,} \, (A-B)\in\calI_1\}.
\]

%%%%%%%%%%%%%%%
\begin{proposition}\lb{PA.7A} $\xi(\calI_1)$ is the set of $L^1(\bbR;d\lambda)$-elements of compact support.
\end{proposition}
%%%%%%%%%%%%%%%
\begin{proof} Since $A, B$ are bounded and self-adjoint, any $\xi (\dott; A,B)  \in \xi(\calI_1)$
necessarily lies in $L^1(\bbR;d\lambda)$ and has compact support (cf.\ Theorem \ref{TA.5}\,(i) and (iv)).

Next, let $g\in L^1 (\bbR;d\lambda)$ satisfy $0\leq g(\lambda)\leq 1$ and
$\supp(g)\subset (a,b)$ for some $-\infty<a<b<\infty$. Define
\begin{equation} \lb{A.37c}
G(z) = \exp\biggl( \f{1}{\pi} \int_{a}^b \f{g(\lambda)\, d\lambda}{\lambda-z} \biggr),
\quad \Ima (z)>0.
\end{equation}
Then $G$ satisfies the following items (i)--(iii):
\begin{SL}
\item[(i)] $\Ima (G(z))>0$ (for $\Ima (z)>0$) since
\[
0\leq \Ima \biggl(\int_{a}^b \f{g(\lambda)\, d\lambda}{\lambda-z}\biggr)
\leq \Ima \bigg(\int_{a}^b \f{d\lambda}{\lambda-z}\bigg) \leq \pi
\]
on account of $0\leq g \leq 1$.

\item[(ii)] $\Ima (G(\lambda+i0))=0$ if $\lambda<a$ or $\lambda>b$.

\item[(iii)] $G(z)\to 1$ as $\Ima (z)\to \infty$ since $g\in L^1(\bbR;d\lambda)$. It follows that
there is $\alpha >0$ and a probability measure $d\mu$ on $[a,b]$ with
\begin{equation}\lb{A.37b}
G(z)= 1+\alpha \int_a^b \f{d\mu(\lambda)}{\lambda-z}.
\end{equation}
\end{SL}
Let $B$ be multiplication by $\lambda$ on $L^2 ((a,b);d\mu)$, $\varphi$ is the function $1$ in
$L^2 ((a,b);d\mu)$ and $A=B+\alpha (\varphi, \dott)\varphi$. Then, by \eqref{A.5},
\eqref{A.6},
\eqref{A.7}, and \eqref{A.14}, $\xi(\lambda;A,B)=\pi^{-1} g(\lambda)$ for a.e.\
$\lambda\in(a,b)$, and
$\alpha =\pi^{-1} \int_{a}^b g(\lambda)\, d\lambda$. Thus, we have the theorem if
$0\leq g \leq 1$ or (by interchanging $A$
and $B$) if $0\geq g \geq -1$. Since any $L^1(\bbR;d\lambda)$-function is a sum of such
$g$'s converging in $L^1(\bbR;d\lambda)$ (simple functions are dense in
$L^1(\bbR;d\lambda)$), we obtain the general result.
\end{proof}
%%%%%%%%%%%%%%%%%
We note that a similar result for the finite rank case can be found in \cite{KrYa}.

Finally, we prove invariance of the absolutely continuous spectrum under trace class
perturbations using the KrSSF and perturbation determinants, that is, without directly
relying on elements from scattering theory.

We start with the following observations:

%%%%%%%%%%%%%
\begin{lemma} \lb{lA.8}
Let $A,B$ be bounded self-adjoint operators with $X=(A-B)$ of rank one. Then,
$$
\sigma_{\rm ac}(A) = \sigma_{\rm ac}(B)
$$
and
$$
\xi(\lambda;A,B)\in\{-1, 0, 1\}
$$
for a.e.\ $\lambda\in\bbR\backslash \sigma_{\rm ac}(B)$.
\end{lemma}
%%%%%%%%%%%%%

%%%%%%%%%%%%%
\begin{proof}
$\xi(\lambda;A,B)\in\{-1, 0, 1\}$ follows from \eqref{A.5}--\eqref{A.7}, \eqref{A.11}, and
\eqref{A.12}. $\sigma_{\rm ac}(A) = \sigma_{\rm ac}(B)$ follows in the usual manner by
computing the normal boundary values to the real axis of the imaginary part of $F_\alpha$
in terms of that of $F_0$ using \eqref{A.8}.
\end{proof}
%%%%%%%%%%%%%

%%%%%%%%%%%%%
\begin{lemma} \lb{lA.9}
Let $A,B$ be bounded self-adjoint operators with $X=(A-B)\in\calI_1$. Then
for  a.e.\ $\lambda\in\bbR\backslash \sigma_{\rm ac}(B)$ one has
$$
\lim_{\varepsilon\downarrow 0}
\det (I + X(B-\lambda-i \varepsilon)^{-1}) \in \bbR.
$$
\end{lemma}
%%%%%%%%%%%%%
\begin{proof}
By \eqref{A.31}, it suffices to prove that
\begin{equation}
\xi(\dott;A,B)\in\bbZ \, \text{ a.e.\ on $\bbR\backslash \sigma_{\rm ac}(B)$.}  \lb{A.37}
\end{equation}
Introducing
$$
X= \sum_{n=1}^\infty x_n (\phi_n, \dott)\phi_n, \quad  X_0=0, \;
X_N = \sum_{n=1}^N x_n (\phi_n, \dott)\phi_n, \; N\in\bbN,
$$
the rank-by-rank construction of $\xi(\dott;A,B)$ alluded to in the proof of
Theorem~\ref{TA.5} yields the $L^1(\bbR; d\lambda)$-convergent series
\begin{equation}
\xi(\dott;A,B) = \sum_{n=1}^\infty \xi(\dott;B+X_n,B+X_{n-1}).  \lb{A.36a}
\end{equation}
By Lemma~\ref{lA.8}, each term in the above series is
integer-valued a.e.\ on $\bbR\backslash \sigma_{\rm ac}(B)$ and hence so is the
left-hand side of \eqref{A.36a}, which yields \eqref{A.37}.
\end{proof}
%%%%%%%%%%%%

%%%%%%%%%%%%
\begin{lemma} \lb{lA.10}
Let $A,B$ be bounded self-adjoint operators in the Hilbert space $\calH$ with
$X=(A-B)\in\calI_1$ and $\varphi\in\calH$, $\|\varphi\|=1$. Denote
$P_{\varphi}=(\varphi, \dott)\varphi$. Then,
\begin{equation}
1 - (\varphi, (B-z)^{-1} \varphi) = \frac{\det (I - (X + P_{\varphi}) (A-z)^{-1})}
{\det (I - X(A-z)^{-1})}, \quad z\in\bbC_+.  \lb{A.38}
\end{equation}
\end{lemma}
%%%%%%%%%%%%
\begin{proof}
One computes
\begin{align}
I - P_{\varphi} (B-z)^{-1} & = (B - P_{\varphi} -z)(A-z)^{-1} (A-z) (B-z)^{-1}  \no \\
& = (B - P_{\varphi} -z)(A-z)^{-1} [(B-z)(A-z)^{-1}]^{-1}   \no \\
& = [I - (X+P_{\varphi})(A-z)^{-1}] [I - X (A-z)^{-1}]^{-1}.    \lb{A.41}
\end{align}
Taking determinants in \eqref{A.41} then yields
\begin{align*}
 \frac{\det (I - (X + P_{\varphi})(A-z)^{-1})}{\det (I - X (A-z)^{-1})}
 & = \det(I - P_{\varphi}(B-z)^{-1} )  \\
 & = 1 - (\varphi, (B-z)^{-1} \varphi). \qedhere
\end{align*}
\end{proof}
%%%%%%%%%%%%%

%%%%%%%%%%%%%
\begin{theorem} \lb{tA.11}
Let $A,B$ be bounded self-adjoint operators in the Hilbert space $\calH$ with
$(A-B)\in\calI_1$. Then,
$$
\sigma_{\rm ac}(A) = \sigma_{\rm ac}(B).
$$
\end{theorem}
%%%%%%%%%%%%%
\begin{proof}
By symmetry between $A$ and $B$, it suffices to prove $\sigma_\ac(B) \subseteq \sigma_\ac(A)$. 
Suppose to the contrary that there exists a set $\calE \subseteq \sigma_\ac(B)$ such that 
$\abs{\calE}>0$ and $\calE \cap \sigma_\ac(A) = \emptyset$. Choose an element 
$\varphi\in\calH$ such that $\lim_{\varepsilon\downarrow 0} \Ima ( (\varphi, (B-\lambda - 
i \varepsilon)^{-1})\varphi)>0$ for a.e.\ $\lambda\in\calE$. Thus, for a.e.\ $\lambda\in\calE$, 
the imaginary part of the limit $z\to \lambda+i0$ of the left-hand side of \eqref{A.38} is 
nonzero. On the other hand, by  Lemma~\ref{lA.9}, the right-hand side of \eqref{A.38} is 
real for a.e.\ $\lambda\in\calE$, a contradiction.
\end{proof}
%%%%%%%%%%%%%

%%%%%%%%%%%%%
\begin{remark}
Employing $\det(I-A)={\det}_2(I-A)e^{\tr(A)}$ for $A\in\calI_1$, and using an approximation 
of Hilbert--Schmidt operators by trace class operators in the norm $\|\cdot\|_{\calI_2}$, 
one rewrites \eqref{A.38} in the case where $X=(A-B)\in\calI_2$ as
\begin{align}
& 1 - (\varphi, (B-z)^{-1} \varphi) = \frac{{\det}_2 (I - (X + P_{\varphi}) (A-z)^{-1})}
{{\det}_2 (I - X(A-z)^{-1})} e^{(\varphi, (A-z)^{-1} \varphi)},   \no \\
&  \hspace*{9.3cm}  z\in\bbC_+.  \lb{A.42}
\end{align}

Since in the proof of Theorem~\ref{tA.11} one assumes $\calE\subseteq\sigma_\ac(B)$, 
$\abs{\calE}>0$, and $\calE\cap\sigma_\ac(A)=\emptyset$, one concludes that
$$
(\varphi, (A-\lambda - i 0)^{-1} \varphi) \, \text{ is real-valued for
a.e.\ $\lambda\in\calE$}. 
$$
Moreover, if the boundary values of ${\det}_2 (I - X(A-\lambda -i 0)^{-1})$ exist 
for a.e.\ $\lambda\in \sigma_\ac(B)$, by \eqref{A.42}, so do those of ${\det}_2 
(I - (X + P_{\varphi}) (A-\lambda - i 0)^{-1})$. Hence, if one can assert real-valuedness of
\begin{equation}
\frac{{\det}_2 (I - (X + P_{\varphi}) (A-\lambda - i 0)^{-1})}
{{\det}_2 (I - X(A- \lambda - i 0)^{-1})} \,
\text{ for a.e.\ $\lambda\in\sigma_\ac(B)$},    \lb{A.43}
\end{equation}
using input from some other sources, one can follow the proof of Theorem~\ref{tA.11}
step by step to obtain invariance of the a.c.\ spectrum.

In the special case of Schr\"odinger (and similarly for Jacobi) operators with real-valued
potentials $V\in L^p([0,\infty))$, $p\in [1,2]$, the existence of the boundary values of
${\det}_2 (I - X(A-\lambda -i 0)^{-1})$ is indeed known due to Christ--Kiselev
\cite{CK} (for $p\in[1,2)$ using some heavy machinery) and Killip--Simon \cite{KS} 
(for $p=2$). We will return to this circle of ideas in \cite{GPS}.
\end{remark}
%%%%%%%%%%%%%

%%%%%%%%%%%%%%%%%%%%%%%%%%%%%%%%%%%%

\end{document}